\documentclass[12pt]{amsart}
\usepackage{amssymb}

\let\frak\mathfrak
\let\Bbb\mathbb

\def\>{\relax\ifmmode\mskip.666667\thinmuskip\relax\else\kern.111111em\fi}
\def\<{\relax\ifmmode\mskip-.333333\thinmuskip\relax\else\kern-.0555556em\fi}
\def\vsk#1>{\vskip#1\baselineskip}
\def\vv#1>{\vadjust{\vsk#1>}\ignorespaces}
\def\vvn#1>{\vadjust{\nobreak\vsk#1>\nobreak}\ignorespaces}
\def\vvgood{\vadjust{\penalty-500}}
\let\alb\allowbreak

\def\sskip{\par\vskip.2\baselineskip plus .05\baselineskip}
\let\Medskip\medskip
\def\medskip{\par\Medskip}
\let\Bigskip\bigskip
\def\bigskip{\par\Bigskip}

\let\Maketitle\maketitle
\def\maketitle{\hrule height0pt\vskip-\baselineskip
\Maketitle\thispagestyle{empty}\let\maketitle\empty}

\newtheorem{thm}{Theorem}[section]
\newtheorem{cor}[thm]{Corollary}
\newtheorem{lem}[thm]{Lemma}
\newtheorem{prop}[thm]{Proposition}

\numberwithin{equation}{section}

\theoremstyle{definition}
\newtheorem*{rem}{Remark}
\newtheorem*{example}{Example}

\let\mc\mathcal
\let\nc\newcommand

\nc{\on}{\operatorname}
\nc{\Z}{{\mathbb Z}}
\nc{\C}{{\mathbb C}}
\nc{\N}{{\mathbb N}}
\nc{\pone}{{\mathbb C}{\mathbb P}^1}
\nc{\arr}{\rightarrow}
\nc{\larr}{\longrightarrow}
\nc{\al}{\alpha}
\nc{\W}{{\mc W}}
\nc{\la}{\lambda}
\nc{\su}{\widehat{{\mathfrak sl}}_2}
\nc{\g}{{\mathfrak g}}
\nc{\h}{{\mathfrak h}}
\nc{\m}{{\mathfrak m}}
\nc{\n}{{\mathfrak n}}
\nc{\Gm}{\Gamma}
\nc{\La}{\Lambda}
\nc{\gl}{\widehat{\mathfrak{gl}_2}}
\nc{\bi}{\bibitem}
\nc{\om}{\omega}
\nc{\Res}{\on{Res}}
\nc{\gm}{\gamma}
\nc{\Om}{\Omega}

\def\fratop{\genfrac{}{}{0pt}1}
\def\satop#1#2{\fratop{\scriptstyle#1}{\scriptstyle#2}}

\def\Mu{\mathrm M}
\def\z{\mathfrak z}
\def\ev{\mbox{\sl ev}}
\def\ch{\on{ch}}
\def\gr{\on{gr}}
\def\Ann{\on{Ann}}
\def\End{\on{End\>}}

\def\Res{\on{Res}}
\def\rdet{\on{rdet}}
\def\Wr{\on{Wr}}

\def\lab{{\bs{\bar\la}}}

\def\lba{{\bs\la\>,\bs a}}

\def\Llb{{\bs\La,\bs\la\>,\bs b}}
\def\Lbl{{\bs\La,\bs b}}

\def\Wrnb{\Wr_{\bs n,\bs b}}
\def\Wrnbi{\Wrnb^{-1}}

\def\B{{\mc B}}
\def\D{{\mc D}}

\def\L{{\mc L}}
\def\M{{\mc M}}
\def\O{{\mc O}}
\def\Q{{\mc Q}}
\def\V{{\mc V}}
\def\Dt{\Tilde\D}

\def\flati{\def\=##1{\rlap##1\hphantom b)}}

\let\dl\delta
\let\Dl\Delta

\let\si\sigma

\let\Sig\varSigma
\let\Tht\Theta

\let\Tilde\widetilde
\let\der\partial

\let\ge\geqslant
\let\geq\geqslant
\let\le\leqslant
\let\leq\leqslant

\let\bat\bar

\nc{\gln}{\mathfrak{gl}_N}
\nc{\sln}{\mathfrak{sl}_N}

\def\glnt{\gln[t]}
\def\Ugln{U(\gln)}
\def\Uglnt{U(\glnt)}

\def\slnt{\sln[t]}

\def\beq{\begin{equation}}
\def\eeq{\end{equation}}
\def\be{\begin{equation*}}
\def\ee{\end{equation*}}

\nc{\bean}{\begin{eqnarray}}
\nc{\eean}{\end{eqnarray}}
\nc{\bea}{\begin{eqnarray*}}
\nc{\eea}{\end{eqnarray*}}
\nc{\bs}{\boldsymbol}
\nc{\Ref}[1]{{\rm(\ref{#1})}}
\nc{\glN}{\mathfrak{gl}_N}
\nc{\glNt}{\mathfrak{gl}_N[t]}
\nc{\s}{sing}
\nc{\R}{\Bbb R}

\nc{\Oml}{{\Om_{\bs\la}}}
\nc{\OmLb}{{\Om_{\bs\La,\bs\la,\bs b}}}
\nc{\Ol}{{\mc O_{\bs\la}}}
\nc{\OLb}{{\mc O_{\bs\La,\bs\la,\bs b}}}
\nc{\VSl}{{(\V^S)_{\bs\la}}}
\nc{\Bl}{{\B_{\bs\la}}}
\nc{\Ml}{{\mc M_{\bs\la}}}
\nc{\Mlb}{{\mc M_{\bs\La,\bs\la,\bs b}}}
\nc{\Blb}{{\B_{\bs\La,\bs\la,\bs b}}}
\nc{\Omn}{{\Omega_{\bs n,\bs b,\bs K}}}
\nc{\Omlb}{{\bar\Om_{\bs\la}}}
\nc{\ep}{\epsilon}
\def\Dlnb{\Dl_{\bs n,\bs b,\bs K}}
\nc{\Dlb}{\Dl_{\bs\La,\bs\la,\bs b,\bs K}}

\begin{document}

\title[Spaces of quasi-exponentials and representations of {\small$\gln$}]
{Spaces of quasi-exponentials and representations of {\large$\gln$}}

\author[E.\,Mukhin, V.\,Tarasov, and A.\,Varchenko]
{E.\,Mukhin$\>^{*,1}$, V.\,Tarasov$\>^{\star,*,2}$,
and A.\,Varchenko$\>^{\diamond,3}$}

\thanks{${}^1$\ Supported in part by NSF grant DMS-0601005}
\thanks{${}^2$\ Supported in part by RFFI grant 05-01-00922}
\thanks{${}^3$\ Supported in part by NSF grant DMS-0555327}

\maketitle

\begin{center}
{\it $^\star\<$Department of Mathematical Sciences
Indiana University\,--\>Purdue University Indianapolis\\
402 North Blackford St, Indianapolis, IN 46202-3216, USA\/}

\medskip
{\it $^*\<$St.\,Petersburg Branch of Steklov Mathematical Institute\\
Fontanka 27, St.\,Petersburg, 191023, Russia\/}

\medskip
{\it $^\diamond\<$Department of Mathematics, University of North Carolina
at Chapel Hill\\ Chapel Hill, NC 27599-3250, USA\/}
\end{center}

\medskip
\begin{abstract}
We consider the action of the Bethe algebra $\B_{\bs K}$ on
$(\otimes_{s=1}^kL_{\bs\la^{(s)}})_{\bs\la}$, the weight subspace of
weight $\bs\la$ of the tensor product of $k$ polynomial irreducible
$\gln$-modules with highest weights $\bs\la^{(1)},\dots,\bs\la^{(k)}$,
respectively. The Bethe algebra depends on $N$ complex numbers $\bs
K=(K_1,\dots,K_N)$.  Under the assumption that $K_1,\dots,K_N$ are
distinct, we prove that the image of $\B_{\bs K}$ in \
End\,$((\otimes_{s=1}^kL_{\bs\la^{(s)}})_{\bs\la})$ is isomorphic to
the algebra of functions on the intersection of $k$ suitable Schubert
cycles in the Grassmannian of $N$-dimensional spaces of
quasi-exponentials with exponents $\bs K$.  We also prove that the
$\B_{\bs K}$-module $(\otimes_{s=1}^kL_{\bs\la^{(s)}})_{\bs\la}$ is
isomorphic to the coregular representation of that algebra of
functions.  We present a Bethe ansatz construction identifying the
eigenvectors of the Bethe algebra with points of that intersection of
Schubert cycles.
\end{abstract}

\section{Introduction}
It has been proved recently in \cite{MTV6} that the eigenvectors of the Bethe
algebra of the $\glN$ Gaudin model are in a bijective correspondence with
$N$-th order Fuchsian differential operators with polynomial kernel and
prescribed singularities. In this paper we construct a variant of this correspondence.

The Bethe algebra considered in \cite{T}, \cite{MTV6} admits a deformation $\B_{\bs K}$
depending on $N$ complex parameters $\bs K=(K_1,\dots,K_N)$, see \cite{CT},\cite{MTV1}.
Under the assumption that $K_1,\dots,K_N$ are distinct, 
we consider the Bethe algebra $\B_{\bs K}$ acting on 
$(\otimes_{s=1}^kL_{\bs\la^{(s)}})_{\bs\la}$, the weight subspace of weight
$\bs\la$ of the tensor product
of  $k$ polynomial irreducible $\gln$-modules with highest weights 
$\bs\la^{(1)},\dots,\bs\la^{(k)}$, respectively. 
We prove that the image of $\B_{\bs K}$ in
\ End\,$((\otimes_{s=1}^kL_{\bs\la^{(s)}})_{\bs\la})$
is isomorphic to 
the algebra of functions on the intersection of
$k$ suitable Schubert cycles in the Grassmannian of $N$-dimensional 
spaces of quasi-exponentials with exponents $\bs K$.
We prove that the $\B_{\bs K}$-module
$(\otimes_{s=1}^kL_{\bs\la^{(s)}})_{\bs\la}$ is isomorphic to the coregular
representation of that algebra of functions.

We present a Bethe ansatz construction identifying the eigenvectors of the Bethe algebra
$\B_{\bs K}$ with points of that intersection of Schubert cycles, cf. \cite{MTV7}.

Thus, we show that the eigenvectors of $\B_{\bs K}$ are in a bijective
correspondence with suitable $N$-th order differential operators with
quasi-exponential kernel and prescribed singularities.  

This
correspondence reduces the multidimensional problem of the
diagonalization of the Bethe algebra action  to the
one-dimensional problem of finding the corresponding differential
operators.

A separation of variables in a quantum integrable
model is a reduction of a multidimensional spectral problem to a
suitable one-dimensional problem, see for example Sklyanin's
separation of variables in the $\frak{gl}_2$ Gaudin model.  In that
respect, our correspondence can be viewed as ``a separation of
variables'' in the $\gln$ Gaudin model associated with
$\B_{\bs K}$, cf. \cite{MTV7}.

The results of this paper and of \cite{MTV6} are in the spirit of the $\gln$ geometric 
Langlands correspondence, which, in particular, relates suitable
commutative algebras of linear
operators acting on $\gln$-modules with properties of  schemes of suitable
$N$-th order differential operators.

\medskip

The paper is organized as follows. In Section~\ref{alg sec}, we discuss
representations of the current algebra $\glnt$, in particular, Weyl modules.
We introduce the Bethe algebra as a subalgebra of $U(\glnt)$ in
Section~\ref{bethesec}.
In Section \ref{Spaces of quasi-exponentials and Wronski map} we introduce the
space of quasi-exponentials and discuss properties of the algebra of functions
 on that space.

 In Section~\ref{schubert}, we introduce a collection of (Schubert) subvarieties in the
affine space $\Omega_{\bs\la}$ of collections
of $N$ quasi-exponentials and consider the algebra of functions on the intersection
of the subvarieties.
We prove the main results of the paper, Theorems~\ref{first}, \ref{first1},
\ref{second}, and ~\ref{third} in Section~\ref{iso sec}. 
Section~\ref{app sec} describes applications.

\medskip

The results of this paper are related to the results on
the $\gln$-opers with an irregular singularity in the recent paper
\cite{FFR}.

\subsection*{Acknowledgments}
We thank D.\,Arinkin for useful discussions and  L.\,Rybnikov for
sending us preprint \cite{FFR}.

\section{Representations of current algebra $\glnt$}
\label{alg sec}
\subsection{Lie algebra $\gln$}
Let $e_{ij}$, $i,j=1,\dots,N$, be the standard generators of the Lie algebra
$\gln$ satisfying the relations
$[e_{ij},e_{sk}]=\dl_{js}e_{ik}-\dl_{ik}e_{sj}$. We identify the Lie algebra
$\sln$ with the subalgebra in $\gln$ generated by the elements
$e_{ii}-e_{jj}$ and $e_{ij}$ for $i\ne j$, $i,j=1,\dots,N$.
We denote by $\h\subset \glN$ the subalgebra generated by
$e_{ii},\,i=1,\dots,N$.

The subalgebra $\z_N\subset\gln$ generated by the element
$\sum_{i=1}^Ne_{ii}\,$ is central. The Lie algebra $\gln$ is canonically
isomorphic to the direct sum $\sln\oplus\z_N$.

Given an $N\times N$ matrix $A$ with possibly noncommuting entries $a_{ij}$,
we define its {\it row determinant\/} to be
\beq
\rdet A\,=
\sum_{\;\si\in S_N\!} (-1)^\si\,a_{1\si(1)}a_{2\si(2)}\dots a_{N\si(N)}\,.
\eeq

Let $Z(x)$ be the following polynomial in a variable $x$ with coefficients
in $\Ugln$:
\vvn.3>
\beq
\label{Zx}
Z(x)\,=\,\rdet\left( \begin{matrix}
x-e_{11} & -\>e_{21}& \dots & -\>e_{N1}\\
-\>e_{12} &x+1-e_{22}& \dots & -\>e_{N2}\\
\dots & \dots &\dots &\dots \\
-\>e_{1N} & -\>e_{2N}& \dots & x+N-1-e_{NN}
\end{matrix}\right).
\vv.3>
\eeq
The next statement was proved in \cite{HU}, see also \cite[Section~2.11]{MNO}.

\begin{thm}
\label{Zcent}
The coefficients of the polynomial\/ $\,Z(x)-x^N$ are free generators
of the center of\/ $\Ugln$.
\qed
\end{thm}

Let $M$ be a $\gln$-module. A vector $v\in M$ has weight
$\bs\la=(\la_1,\dots,\la_N)\in\C^N$ if $e_{ii}v=\la_iv$ for $i=1,\dots,N$.
A vector $v$ is called {\it singular\/} if $e_{ij}v=0$ for $1\le i<j\le N$.
If $v$ is a singular of weight $\bs\la$, then
\vvn-.5>
\beq
\label{Zxv}
Z(x)\,v\,=\,\prod_{i=1}^N\,(x-\la_i+i-1)\cdot v\,.
\eeq

We denote by $(M)_{\bs\la}$ the subspace of $M$ of weight $\bs\la$,
by $(M)^{sing}$ the subspace of $M$ of all singular vectors and by
$(M)^{sing}_{\bs\la}$ the subspace of $M$ of all singular vectors
of weight $\bs\la$.

Denote by $L_{\bs\la}$ the irreducible finite-dimensional $\gln$-module with
highest weight $\bs\la$. Any finite-dimensional $\gln$-module $M$ is isomorphic
to the direct sum $\bigoplus_{\bs\la}L_{\bs\la}\otimes(M)_{\bs\la}^{sing}$,
where the spaces $(M)_{\bs\la}^{sing}$ are considered as trivial
$\gln$-modules.

The $\gln$-module $L_{(1,0,\dots,0)}$ is the standard $N$-dimensional vector
representation of $\gln$. We denote it by $V$. We choose a highest weight
vector in $V$ and denote it by $v_+$.

A $\gln$-module $M$ is called polynomial if it is isomorphic to a submodule of
$V^{\otimes n}$ for some $n$.

A sequence of integers $\bs\la=(\la_1,\dots,\la_N)$ such that
$\la_1\ge\la_2\ge\dots\ge\la_N\ge0$ is called a {\it partition with at most
$N$ parts\/}. Set $|\bs\la|=\sum_{i=1}^N\la_i$. Then it is said that $\bs\la$
is a partition of $|\bs\la|$.

\sskip
The $\gln$-module $V^{\otimes n}$ contains the module $L_{\bs\la}$
if and only if $\bs\la$ is a partition of $n$ with at most $N$ parts.

\sskip
For a Lie algebra $\g\,$, we denote by $U(\g)$ the universal enveloping algebra
of $\g$.

\subsection{Current algebra $\glnt$}
Let $\glnt=\gln\otimes\C[t]$ be the Lie algebra of $\gln$-valued polynomials
with the pointwise commutator. We call it the {\it current algebra\/}.
We identify the Lie algebra $\gln$ with the subalgebra $\gln\otimes1$
of constant polynomials in $\glnt$. Hence, any $\glnt$-module has the canonical
structure of a $\gln$-module.

The standard generators of $\glnt$ are $e_{ij}\otimes t^r$, $i,j=1,\dots,N$,
$r\in\Z_{\ge0}$. They satisfy the relations
$[e_{ij}\otimes t^r,e_{sk}\otimes t^p]=
\dl_{js}e_{ik}\otimes t^{r+p}-\dl_{ik}e_{sj}\otimes t^{r+p}$.

The subalgebra $\z_N[t]\subset\glnt$ generated by the elements
$\sum_{i=1}^Ne_{ii}\otimes t^r$, $r\in\Z_{\ge0}$, is central.
The Lie algebra $\glnt$ is canonically isomorphic to the direct sum
$\slnt\oplus\z_N[t]$.

It is convenient to collect elements of $\glnt$ in generating series
of a  variable $u$. For $g\in\gln$, set
\vvn-.1>
\be
g(u)=\sum_{s=0}^\infty (g\otimes t^s)u^{-s-1}.
\vv.4>
\ee

For each $a\in\C$, there exists an automorphism $\rho_a$ of $\glnt$,
\;$\rho_a:g(u)\mapsto g(u-a)$. Given a $\glnt$-module $M$, we denote by $M(a)$
the pull-back of $M$ through the automorphism $\rho_a$. As $\gln$-modules,
$M$ and $M(a)$ are isomorphic by the identity map.

For any $\glnt$-modules $L,M$ and any $a\in\C$, the identity map
$(L\otimes M)(a)\to L(a)\otimes M(a)$ is an isomorphism of $\glnt$-modules.

We have the evaluation homomorphism,
${\ev:\glnt\to\gln}$, \;${\ev:g(u) \mapsto g\>u^{-1}}$.
Its restriction to the subalgebra $\gln\subset\glnt$ is the identity map.
For any $\gln$-module $M$, we denote by the same letter the $\glnt$-module,
obtained by pulling $M$ back through the evaluation homomorphism. For each
$a\in\C$, the $\glnt$-module $M(a)$ is called an {\it evaluation module\/}.

If $b_1,\dots,b_n$ are distinct complex numbers and $L_1,\dots,L_n$ are
irreducible finite-dimensional $\gln$-modules, then the $\glnt$-module
$\otimes_{s=1}^n L_s(b_s)$ is irreducible.

We have a natural $\Z_{\ge0}$-grading on $\glnt$ such that for any $g\in\gln$,
the degree of $g\otimes t^r$ equals $r$. We set the degree of $u$ to be $1$.
Then the series $g(u)$ is homogeneous of degree $-1$.

A $\glnt$-module is called {\it graded\/} if it
has a $\Z_{\ge0}$-grading compatible with the grading of $\glnt$.
Any irreducible graded $\glnt$-module is isomorphic to an evaluation module
$L(0)$ for some irreducible $\gln$-module $L$, see \cite{CG}.

Let $M$ be a $\Z_{\ge0}$-graded space with finite-dimensional homogeneous
components. Let $M_j\subset M$ be the homogeneous component of degree $j$.
We call the formal power series in a variable $q$,
\vvn-.2>
\be
\ch_M(q)\,=\,\sum_{j=0}^\infty\,(\dim M_j)\,q^j\,,
\vv-.2>
\ee
the {\it graded character\/} of $M$.

\subsection{Weyl modules}
\label{secweyl}
Let $W_m$ be the $\glnt$-module generated by a vector $v_m$ with the
defining relations:
\vvn-.2>
\begin{alignat*}2
e_{ii} &(u)v_m=\>\dl_{1i}\,\frac mu\,v_m\,, && i=1,\dots,N\,,
\\[4pt]
e_{ij} & (u)v_m=\>0\,, && 1\le i<j\le N\,,
\\[4pt]
(e_{ji} &{}\otimes1)^{m\dl_{1i}+1}v_m=\>0\,,\qquad && 1\le i<j\le N\,.
\\[-12pt]
\end{alignat*}
As an $\slnt$-module, the module $W_m$ is isomorphic to the Weyl module from
\cite{CL}, \cite{CP}, corresponding to the weight $m\om_1$, where $\om_1$ is
the first fundamental weight of $\sln$. Note that $W_1=V(0)$.

\begin{lem}
\label{weyl}
The module $W_m$ has the following properties.
\begin{enumerate}
\item[(i)] The module \>$W_m$ has a unique grading such that\/ \>$W_m$ is
a graded $\glnt$-module and the degree of\/ $v_m$ equals\/ $0$.
\item[(ii)] As a $\gln$-module, $W_m$ is isomorphic to $V^{\otimes m}$.
\item[(iii)] A $\glnt$-module $M$ is an irreducible subquotient of\/ $W_m$
if and only if\/ $M$ has the form $L_{\bs\la}(0)$,
where\/ $\bs\la$ is a partition of\/ $m$ with at most $N$ parts.
\end{enumerate}
\end{lem}
\begin{proof}
The first two properties are proved in \cite{CP}. The third property follows
from the first two. 
\vvgood
\end{proof}

For each $b\in\C$, the $\glnt$-module $W_m(b)$ is cyclic with a cyclic vector
$v_m$.

\begin{lem}[\cite{MTV6}]
\label{weylb}
The module $W_m(b)$ has the following properties.
\begin{enumerate}
\item[(i)] As a $\gln$-module, $W_m(b)$ is isomorphic to $V^{\otimes m}$.
\item[(ii)] A $\glnt$-module $M$ is an irreducible subquotient of\/ $W_m(b)$
if and only if\/ $M$ has the form $L_{\bs\la}(b)$,
where\/ $\bs\la$ is a partition of\/ $m$ with\/ $N$ parts.
\item[(iii)] For any natural numbers $n_1,\dots,n_k$ and distinct complex
numbers $b_1,\dots,b_k$, the $\glnt$-module $\otimes_{s=1}^k W_{n_s}(b_s)$ is
cyclic with a cyclic vector $\otimes_{s=1}^kv_{n_s}$.
\item[(iv)] Let\/ $M$ be a cyclic finite-dimensional $\glnt$-module with
a cyclic vector $v$ satisfying $e_{ij}(u)v=0$ for $1\le i<j\le N$, and
$e_{ii}(u)v=\dl_{1i}(\sum_{s=1}^k n_s/(u-b_s))v$ for $i=1,\dots,N$.
Then there exists a surjective $\glnt$-module
homomorphism  
$\otimes_{s=1}^kW_{n_s}(b_s) \,\to \,M$ sending\/ $\otimes_{s=1}^kv_{n_s}$ to $v$.
\end{enumerate}
\qed
\end{lem}

Given sequences $\bs n=(n_1,\dots,n_k)$ of natural numbers and
$\bs b=(b_1,\dots,b_k)$ of distinct complex numbers, we call the $\glnt$-module
$\otimes_{s=1}^k W_{n_s}(b_s)$ the {\it Weyl module associated with\/ $\bs n$
and\/ $\bs b$}.

\begin{cor}[\cite{MTV6}]
\label{weylbk}
A $\glnt$-module $M$ is an irreducible subquotient of\/
$\otimes_{s=1}^k W_{n_s}(b_s)$ if and only if\/ $M$ has the form
$\otimes_{s=1}^kL_{\bs\la^{(s)}}(b_s)$, where
$\bs\la^{(1)},\dots,\bs\la^{(k)}$ are partitions with at most $N$ parts
such that $|\bs\la^{(s)}|=n_s$, \,$s=1,\dots,k$.
\qed
\end{cor}

Consider the $\Z_{\ge0}$-grading of the vector space $W_m$, introduced
in Lemma~\ref{weyl}. Let $W_m^j$ be the homogeneous component of $W_m$
of degree $j$ and $\bar W^j_m=\oplus_{r\ge j}W_m^{r}$.
Since the $\glnt$-module $W_m$ is graded and $W_m=W_m(b)$ as vector spaces,
$W_m(b)=\bar W_m^0\supset\bar W_m^1\supset\dots{}$ is a descending filtration
of $\glnt$-submodules. This filtration induces the structure of the associated
graded $\glnt$-module on the vector space $W_m$ which we denote by
$\gr W_m(b)$.

\begin{lem}[\cite{MTV6}]
\label{grWmb}
The $\glnt$-module $\gr W_m(b)$ is isomorphic to the evaluation module\/
$(V^{\otimes m})(b)$.
\vvn-.3>\qed
\end{lem}

The space $\otimes_{s=1}^k W_{n_s}$ has a natural $\Z_{\ge0}^k$-grading,
induced by the gradings on the factors, and the associated descending
$\Z_{\ge0}^k$-filtration by the subspaces $\otimes_{s=1}^k\bar W_{n_s}^{j_s}$,
invariant with respect to the $\glnt$-action on the module
$\otimes_{s=1}^k W_{n_s}(b_s)$. We denote by
$\gr\bigl(\otimes_{s=1}^k W_{n_s}(b_s)\bigr)$ the induced structure of
the associated graded $\glnt$-module on the space $\otimes_{s=1}^k W_{n_s}$.

\begin{lem}
\label{grtensor}
\label{split lem}
The $\glnt$-modules \,$\gr\bigl(\otimes_{s=1}^k W_{n_s}(b_s)\bigr)$
and\/ \,$\otimes_{s=1}^k \gr W_{n_s}(b_s)$ are canonically isomorphic.
\qed
\end{lem}

\subsection{Remark on representations of symmetric group}
\label{repsym}
Let $S_n$ be the group of permutations of $n$ elements. We denote by $\C[S_n]$
the regular representation of $S_n$. Given an $S_n$-module $M$ we denote by
$M^S$ the subspace of all $S_n$-invariant vectors in $M$.

\begin{lem}
\label{elementary}
Let\/ $U$ be a finite-dimensional $S_n$-module.
Then \,$\dim(U\otimes\C[S_n])^S=\dim U$.
\vvn-.3>\qed
\end{lem}

The group $S_n$ acts on the algebra $\C[z_1,\dots,z_n]$ by permuting
the variables. Let $\si_s(\bs z)$, $s=1,\alb\dots,n$, be the $s$-th elementary
symmetric polynomial in $z_1,\dots,z_n$. The algebra of symmetric polynomials
$\C[z_1,\dots,z_n]^S$ is a free polynomial algebra with generators
$\si_1(\bs z),\dots,\alb\si_n(\bs z)$. It is well-known that the algebra
$\C[z_1,\dots,z_n]$ is a free $\C[z_1,\dots,z_n]^S$-module of\linebreak
 rank $n!$,
see~\cite{M}.

Given $\bs a=(a_1,\dots,a_n)\in\C^n$, denote by
$I_{\bs a}\subset C[z_1,\dots,z_n]$ the ideal generated by
the polynomials $\si_s(\bs z)-a_s$, $s=1,\dots, n$.
The ideal $I_{\bs a}$ is $S_n$-invariant.

\begin{lem}
\label{regular}
For any $\bs a\in\C^n$, the $S_n$-representation $\C[z_1,\dots,z_n]/I_{\bs a}$
is isomorphic to the regular representation $\C[S_n]$.
\qed
\end{lem}

\subsection{The $\glnt$-module $\V^S$}
\label{VS}
Let $\V$ be the space of polynomials in $z_1,\dots,z_n$ with coefficients
in $V^{\otimes n}$:
\vvn-.3>
\be
\V\>=\,V^{\otimes n}\<\otimes_{\C}\C[z_1,\dots,z_n]\,.
\vv.2>
\ee
The space $V^{\otimes n}$ is embedded in $\V$ as the subspace of constant
polynomials.

Abusing notation, for any $v\in V^{\otimes n}$ and
$p(z_1,\dots,z_n)\in\C[z_1,\dots,z_n]$, we will write
$p(z_1,\dots,z_n)\,v$ instead of $v\otimes p(z_1,\dots,z_n)$.

We make the symmetric group $S_n$ act on $\V$ by permuting the factors
of $V^{\otimes n}$ and the variables $z_1,\dots,z_n$ simultaneously,
\vvn.2>
\be
\si\bigl(p(z_1,\dots,z_n)\,v_1\otimes\dots\otimes v_n\bigr)\,=\,
p(z_{\si_1},\dots,z_{\si_n})\,
v_{(\si^{-1})_1}\!\otimes\dots\otimes v_{(\sigma^{-1})_n}\,,\qquad\si\in S_n\,.
\kern-3em
\vv.2>
\ee
We denote by $\V^S$ the subspace of $S_n$-invariants in $\V$.

\begin{lem}[\cite{MTV6}]
\label{VSfree}
The space $\V^S$ is a free $\C[z_1,\dots,z_n]^S$-module of rank $N^n$.
\qed
\end{lem}

We consider the space $\V$ as a $\glnt$-module with
the series $g(u)$, \,$g\in\gln$, acting by
\vvn.1>
\beq
\label{action}
g(u)\,\bigl(p(z_1,\dots,z_n)\,v_1\otimes\dots\otimes v_n)\,=\,
p(z_1,\dots,z_n)\,\sum_{s=1}^n
\frac{v_1\otimes\dots\otimes gv_s\otimes\dots\otimes v_n}{u-z_s}\ .
\vv.2>
\eeq

\begin{lem}[\cite{MTV6}]
\label{Uz}
The image of the subalgebra $U(\z_N[t])\subset \Uglnt$ in
$\End(\V)$ coincides with the algebra of operators of multiplication
by elements of\/ $\C[z_1,\dots,z_n]^S$.
\qed
\end{lem}

The $\glnt$-action on $\V$ commutes with the $S_n$-action.
Hence, $\V^S$ is a $\glnt$-submodule of $\V$.

\goodbreak
Consider the grading on $\C[z_1,\dots,z_n]$ such that $\deg z_i=1$ for all
$i=1,\dots,n$. We define a grading on $\V$ by setting $\deg(v\otimes p)=\deg p$
for any $v\in V^{\otimes n}$ and any $p\in\C[z_1,\dots,z_n]$. The grading on
$\V$ induces a natural grading on $\End(\V)$.

\begin{lem}[\cite{MTV6}]
\label{cycl grad}
The $\glnt$-modules\/ $\V$ and $\V^S$ are graded.
\qed
\end{lem}

The following lemma is contained in \cite{K}, see also \cite{MTV6}.

\begin{lem}
\label{cycl=symm}
The $\glnt$-module $\V^S$ is cyclic with a cyclic vector\/ $v_+^{\otimes n}$.
\qed
\end{lem}

\begin{lem}
\label{lem on char of VSl}
For any partition $\bs\la$ of\/ $n$ with\/ at most $N$ parts,
 the
 graded character of the space $(\V^S)_{\bs\la}$ is given by
\vvn-.2>
\bean
\label{forMula}
\ch_{(\V^S)_{\bs\la}}(q)\,=\,{\prod_{i=1}^N\frac1{(q)_{\la_i}}}\;.
\eean
where\/ $\,(q)_a=\prod_{j=1}^a(1-q^j)\,$.
\end{lem}
\begin{proof}
A basis of $(\V^S)_{\bs\la}$ is given by the $S_n$-orbits of the
$V^{\otimes n}$-valued polynomials of the form
\bea
p(z_1,\dots,z_n)\,(v_+)^{\otimes \la_1}\otimes (e_{21}v_+)^{\otimes \la_2}\otimes
\dots \otimes (e_{N1}v_+)^{\otimes \la_N}\ ,
\eea
where $p(z_1,\dots,z_n)$ is a polynomial symmetric with respect to the first
$\la_1$  variables, symmetric with respect to the next
$\la_2$  variables, and so on and finally 
 symmetric with respect to the last $\la_N$ variables.
Clearly the graded character of the space of such polynomials is given by
formula \Ref{forMula}.
\end{proof}

\subsection{Weyl modules as quotients of $\V^S$}
Let $\bs a=(a_1,\dots,a_n)\in\C^n$ be a sequence of complex numbers and
$I_{\bs a}\subset \C[z_1,\dots,z_n]$ the ideal,
defined in Section~\ref{repsym}. Define
\vvn.3>
\beq
\label{IVa}
I^\V_{\bs a}\,=\,I_{\bs a}\V^S\,.
\vv.2>
\eeq
Clearly, ${I^\V_{\bs a}}$ is a $\glnt$-submodule of $\V^S$.

Introduce distinct complex numbers $b_1,\dots,b_k$ and
natural numbers $n_1,\dots,n_k$ by the relation
\vvn-.6>
\beq\label{ab}
\prod_{s=1}^k\,(u-b_s)^{n_s}=\,u^n+\>\sum_{j=1}^n (-1)^j\>a_j\>u^{n-j}.
\vv-.7>
\eeq
Clearly, $\sum_{s=1}^kn_s=n$.

\begin{lem}[\cite{MTV6}]
\label{factor=weyl}
The $\glnt$-modules $\V^S/{I^\V_{\bs a}}$ and \,$\otimes_{s=1}^kW_{n_s}(b_s)$
are isomorphic.
\qed
\end{lem}

\section{Bethe algebra}
\label{bethesec}
\subsection{Universal differential operator}
\label{secbethe}
Let $\bs K=(K_{1},\dots,K_N)$ be a seguence of distinct complex numbers.
Let $\der$ be the operator of differentiation in a variable $u$.
Define the {\it universal differential operator\/} $\D^\B$ by
\vvn.3>
\be
\D^\B=\,\rdet\left( \begin{matrix}
\der-K_{1}-\>e_{11}(u) & -\>e_{21}(u)& \dots & -\>e_{N1}(u)\\
-\>e_{12}(u) &\der-K_{2}-e_{22}(u)& \dots & -\>e_{N2}(u)\\
\dots & \dots &\dots &\dots \\
-\>e_{1N}(u) & -\>e_{2N}(u)& \dots & \der-K_{N}-e_{NN}(u)
\end{matrix}\right).
\vv.2>
\ee
It is a differential operator in the variable $u$, whose coefficients are
formal power series in $u^{-1}$ with coefficients in
$\Uglnt$,
\vvn-.3>
\beq
\label{DB}
\D^\B=\,\der^N+\sum_{i=1}^N\,B_i(u)\,\der^{N-i}\>,
\vv-.5>
\eeq
where
\vvn-.3>
\beq
\label{Bi}
B_i(u)\,=\,\sum_{j=0}^\infty B_{ij}\>u^{-j}
\eeq
and \,$B_{ij}\in\Uglnt$ \,for \,$i=1,\dots,N$, \,$j\geq 0$.

\begin{lem}
We have
\vvn-.4>
\beq
\label{B1}
B_1(u)\,=\,-\sum_{i=1}^N\,\bigl(K_{i}+ e_{ii}(u)\bigr)
\vv-.6>
\eeq
and
\vvn-.4>
\beq
\label{Bii}
\sum_{i=0}^N\,B_{i\>0}\,\al^{N-i}\,=\,\prod_{i=1}^N\,(\al-K_i)\,,
\vv.2>
\eeq
where \,$\al$ is a variable and \>$B_{00}=1$.
\qed
\end{lem}

\begin{lem}
\label{B deg}
The element $B_{ij}\in\Uglnt$ \,$i=1,\dots,N$, $j\ge 1$, is a sum
of homogeneous elements of degrees \,$j-1\>,j-2\>,\dots,\max\>(j-i,0)$\,.
\end{lem}
\begin{proof}
It is straightforward to see that the series $B_i(u)-\si_i(K_1,\dots,K_N)$,
where $\si_i$ is the $i$-th elementary symmetric polynomial, is a sum
of homogeneous series of degrees $-1,\dots,-i$. The lemma follows.
\end{proof}

We call the unital subalgebra of $\Uglnt$ generated by $B_{ij}$,
with $i=1,\dots,N$, $j\geq 0$, the {\it Bethe algebra\/}
and denote it by $\B$.

\begin{thm}[\cite{CT}, \cite{MTV1}]
\label{T-thm}
The algebra $\B$ is commutative. The algebra $\B$ commutes with
the subalgebra $U(\h)\subset \Uglnt$.
\qed
\end{thm}

\subsection{$\B$-modules}
\label{Bmodules}
Let $\bs\La=(\bs\la^{(1)},\dots,\bs\la^{(k)})$ be a sequence of partitions
with at most $N$ parts and $b_1,\dots,b_k$ distinct complex numbers.

\sskip
Let $\>A_1(u),\dots,A_N(u)$ be the Laurent series in $u^{-1}$ obtained
by projecting coefficients of the series $B_i(u)\prod_{s=1}^k(u-b_s)^i$
to $\>\End\bigl(\otimes_{s=1}^kL_{\bs\la^{(s)}}(b_s)\bigr)$.

\sskip
The next lemma was proved in \cite{MTV2}. 

\begin{lem}
\label{Apol}
The series $\>A_1(u),\dots,A_N(u)$ are polynomials in\/ $u$. Moreover,
the operators $A_i(b_s)$, \,$i=1,\dots,N$, \,$s=1,\dots,k$, are proportional
to the identity operator, and
\be
\sum_{i=0}^N\,A_i(b_s)\prod_{j=0}^{N-i-1}(\al-j)\,=\,
\prod_{l=1}^N\,(\al-\la^{(s)}_l-N+l\>)\,,\qquad s=1,\dots,k\,,\kern-4em
\ee
where $\al$ is a variable and $\>A_0(u)=1$.
\end{lem}
\begin{proof}
Consider the homomorphism $\Uglnt\to(\Ugln)^{\otimes k}$,
\vvn.3>
\beq
g(u)\,\mapsto\,\sum_{s=1}^k\,
\frac{1^{\otimes(s-1)}\otimes g\otimes 1^{\otimes(k-s)}}{u-b_s}\;,
\qquad g\in\gln\,.\kern-3em
\eeq
Let $\>\tilde A_i(u)$, \,${i=1,\dots,N}$, be the Laurent series in $u^{-1}$
obtained by projecting coefficients of the series
\vvn.1>
$B_i(u)\prod_{s=1}^k(u-b_s)^i$ to \>$(\Ugln)^{\otimes k}$.
The series $\>\tilde A_1(u),\dots,\alb\tilde A_N(u)$ are polynomials in \>$u$, 
and by a straightforward calculation
\vvn.3>
\be
\sum_{i=0}^N\,\tilde A_i(b_s)\prod_{j=0}^{N-i-1}(\al-j)\,=\,
1^{\otimes(s-1)}\otimes Z(\al-N+1)\otimes 1^{\otimes(k-s)}\,,
\qquad s=1,\dots,k\,,\kern-4em
\ee
where $\>\tilde A_0(u)=1$.
The lemma follows from Theorem~\ref{Zcent} and formula~\Ref{Zxv}.
\end{proof}

Let $\>C_i(u)$, \,${i=1,\dots,N}$, be the Laurent series in $u^{-1}$ obtained
by projecting coefficients of the series $B_i(u)\prod_{s=1}^n(u-z_s)$ to
$\>\End(\V^S)$.

\begin{lem} 
\label{capelli}
The series $\>C_1(u),\dots,C_N(u)$ are polynomials in\/ $u$.
\end{lem}
\begin{proof}
The statement is a corollary of Theorem~2.1 in \cite{MTV3}.
\end{proof}

Set $n_s=|\bs\la^{(s)}|$, \,$s=1,\dots,k$. Let $\>\bat C_i(u)$,
\,${i=1,\dots,N}$, be the Laurent series in $u^{-1}$ obtained by
projecting coefficients of the series $B_i(u)\prod_{s=1}^k(u-b_s)^{n_s}$
to $\>\End\bigl(\otimes_{s=1}^k W_{n_s}(b_s)\bigr)$.

\begin{cor} 
\label{Cpol}
The series $\>\bat C_1(u),\dots,\bat C_N(u)$ are polynomials in\/ $u$.
\end{cor}
\begin{proof}
The claim follows from Lemmas~\ref{factor=weyl} and~\ref{capelli}.
\end{proof}

\begin{cor} 
\label{Apol2}
The products\/ $A_i(u)\prod_{s=1}^k(u-b_s)^{n_s-i}$, $\,i=1,\dots,k$,
are polynomials in\/ $u$.
\end{cor}
\begin{proof}
The claim follows from Lemma~\ref{weylbk} and Corollary~\ref{Cpol}.
\end{proof}

Let $M$ be a $\glnt$-module.
As a subalgebra of $\Uglnt$, the algebra $\B$ acts on $M$.
If $H\subset M$ is a $\B$-invariant subspace,
then we call the image of $\B$
in $\End(H)$
the {\it Bethe algebra associated with $H$\/}. Since $\B$ commutes
with $U(\h)$,
it preserves the weight subspaces $(M)_{\bs\la}$.

In what follows we study the action of the Bethe algebra $\B$ on the following
$\B$-modules: \
$
(\V^S)_{\bs\la}\,,
{}\
(\otimes_{s=1}^k W_{n_s}(b_s))_{\bs\la}
\,,
{}\
(\otimes_{s=1}^kL_{\bs\la^{(s)}}(b_s))_{\bs\la}\, .
$

\section{Spaces of quasi-exponentials and Wronski map}
\label{Spaces of quasi-exponentials and Wronski map}

\subsection{Spaces of quasi-exponentials}
\label{Spaces of quasi-exponentials}

Let $\bs K=(K_1,\dots,K_N)$ be a sequence of distinct complex numbers.
Let $\bs\la$ be a partition of $n$ with at most $N$ parts.
Let $\Om_{\bs\la}$ be the affine $n$-dimensional space with coordinates
$f_{ij},\,i=1,\dots,N,\,j=1,\dots,\la_i$.

Introduce
\beq
\label{basis}
f_i(u)\,=\,e^{K_iu}\,(u^{\la_i} + f_{i1}u^{\la_i-1} + \dots +f_{i\la_i})\,,
\qquad i=1,\dots,N\,.
\eeq
We identify points \>$X\in\Om_{\bs\la}$ with $N$-dimensional complex
vector spaces generated by quasi-exponentials
\beq
\label{Def of space}
f_i(u,X)\,=\,e^{K_iu}\,(u^{\la_1}+f_{i1}(X)u^{\la_1-1}+\dots+
f_{i\la_1}(X))\,,
\qquad i=1,\dots,N\,.
\eeq

\sskip
Denote by $\mc O_{\bs\la}$ the algebra of regular functions on
$\Om_{\bs\la}$. It is the polynomial algebra in the variables $f_{ij}$.
Define a grading on $\O_{\bs\la}$ such that the degree of the generator
$f_{ij}$ equals $j$ for all $(i,j)$.

\begin{lem}
\label{char O}
The graded character of $\O_{\bs\la}$ is given by the formula
\vvn-.2>
\be
\ch_{\mc O_{\bs\la}}(q)\,=\,{\prod_{i=1}^N\frac1{(q)_{\la_i}}}\;.
\vv-1.4>
\ee
\qed
\end{lem}

\subsection{Another realization of $\O_{\bs\la}$}
For arbitrary
functions $g_1(u),\dots,g_N(u)$, introduce the Wronskian by the formula
\vvn-.3>
\be
\Wr(g_1(u),\dots,g_N(u))\,=\,
\det\left(\begin{matrix} g_1(u) & g_1'(u) &\dots & g_1^{(N-1)}(u) \\
g_2(u) & g_2'(u) &\dots & g_2^{(N-1)}(u) \\ \dots & \dots &\dots & \dots \\
g_N(u) & g_N'(u) &\dots & g_N^{(N-1)}(u)
\end{matrix}\right).
\ee
Let $f_i(u)$, $i=1,\dots,N$, be the generating functions given by \Ref{basis}.
We have
\beq
\label{Wr coef}
\Wr(f_1(u),\dots,f_N(u))\,=\,
e^{\>\sum_{i=1}^N K_iu}\!\prod_{1\le i<j\le N}(K_j-K_i)
\ \Bigl(u^n+\sum_{s=1}^n (-1)^s\>\Sig_s\,u^{n-s}\Bigr)\,,
\eeq
where $\Sig_1,\dots,\Sig_n$ are elements of $\O_{\bs\la}$.
Define the differential operator $\D^\O_{\bs\la}$ by
\vvn.2>
\beq
\label{DOla}
\D^\O_{\bs\la}=\,\frac{1}{\Wr(f_1(u),\dots,f_N(u))}\,\rdet
\left(\begin{matrix} f_1(u) & f_1'(u) &\dots & f_1^{(N)}(u) \\
f_2(u) & f_2'(u) &\dots & f_2^{(N)}(u) \\ \dots & \dots &\dots & \dots \\
1 & \der &\dots & \der^N
\end{matrix}\right).
\eeq
It is a differential operator in the variable $u$, whose coefficients are
formal power series in $u^{-1}$ with coefficients in $\O_{\bs\la}$,
\vvn-.4>
\beq
\label{DO}
\D^\O_{\bs\la}=\,\der^N+\sum_{i=1}^N\,F_i(u)\,\der^{N-i}\>,
\vv-.4>
\eeq
where
\vvn-.3>
\beq
\label{Fi}
F_i(u)\,=\,\sum_{j=0}^\infty F_{ij}\>u^{-j}\,,
\eeq
and $F_{ij}\in\O_{\bs\la}$, \,$i=1,\dots,N$, \,$j\geq 0$.

\sskip
Define the {\it characteristic polynomial of the operator $\D^{\O\/}_{\bs\la}$
at infinity} by
\beq
\label{char at infty}
\chi(\al)\,=\,\sum_{i=0}^{N} F_{i0}\,\al^{N-i}\,,
\eeq
where $\al$ is a variable and $F_{00}=1$. 

\begin{lem}
\label{chi}
We have
\vvn-.8>
\begin{align*}
\chi(\al)\, &{}=\,\prod_{i=1}^N\,(\al-K_i)\,,
\\[4pt]
\sum_{i=1}^N F_{i1}\,\al^{N-i}\, &{}=\,
\sum_{i=1}^N\,\la_i\,\prod_{\satop{j=1}{j\ne i}}^N\,(\al-K_j)\,.
\end{align*}
\end{lem}
\begin{proof}
We have $\D^\O_{\bs\la}f_i(u)=0$ for all $i=1,\dots,N$. Taking
the coefficient of $u^{\la_i}$ of the series $e^{-K_iu}\,\D^\O_{\bs\la}f_i(u)$,
we get $\chi(K_i)=0$, for all $i=1,\dots,N$. This implies the first equality.
The second equality follows similarly from considering the coefficient of
$u^{\la_i-1}$ of the series $e^{-K_iu}\,\D^\O_{\bs\la}f_i(u)$.
\end{proof}

\begin{lem}
\label{coef alg}
The functions $F_{ij}\in\O_{\bs\la}$, $i=1,\dots,N$,
$j\geq 0$, generate the algebra $\O_{\bs\la}$.
\end{lem}
\begin{proof}
The coefficient of $u^{\la_i-j-1}$ of the series
$e^{-K_iu}\,\D^\O_{\bs\la}f_i(u)$ has the form
\vvn.3>
\beq
\label{fij}
-\>j\,f_{ij}\,\prod_{\satop{j=1}{j\ne i}}^N\,(K_i-K_j)\,+\,\sum_{l=1}^N\,
\biggl(\>\sum_{r=0}^1\,\sum_{s=0}^{j-1}\,c_{ijlrs}\>F_{lr}\>f_{is}\,+\,
\sum_{r=2}^j\,\sum_{s=0}^{j-r+1}\,c_{ijlrs}\>F_{lr}\>f_{is}\biggr)\,,
\vv.1>
\eeq
where $c_{ijlrs}$ are some numbers. Since $\D^\O_{\bs\la}f_i(u)=0$, we can
express recursively the elements $f_{ij}$ via the elements $F_{lr}$ starting
with $j=1$ and then increasing the second index $j$.
\end{proof}

\subsection{Frobenius algebras, cf. \cite{MTV6} }
\label{comalg}
In this section, we recall some simple facts from commutative algebra. The word
{\it algebra\/} will stand for an associative unital algebra over $\C$.

Let $A$ be a commutative algebra. The algebra $A$ considered as an $A$-module
is called the {\it regular representation \/} of $A$. The dual space $A^*$ is
naturally an $A$-module, which is called the {\it coregular representation\/}.

Clearly, the image of $A$ in $\End(A)$ for the regular representation is
a maximal commutative subalgebra. If $A$ is finite-dimensional, then the image
of $A$ in $\End(A^*)$ for the coregular representation is a maximal commutative
subalgebra as well.

If $M$ is an $A$-module and $v\in M$ is an eigenvector of the $A$-action on $M$
with eigenvalue $\xi_v\in A^*$, that is, $av=\xi_v(a)\>v$ \;for any $a\in A$,
then $\xi_v$ is a character of $A$, that is, $\xi_v(ab)=\xi_v(a)\>\xi_v(b)$.

If an element $v\in A^*$ is an eigenvector of the coregular action of $A$,
then $v$ is proportional to the character $\xi_v$. Moreover, each character
$\xi\in A^*$ is an eigenvector of the coregular action of $A$ and
the corresponding eigenvalue equals $\xi$.

A nonzero element $\xi\in A^*$ is proportional to a character if and only if
$\,\ker\xi\subset A$ is an ideal. Clearly, $A/\ker\xi\simeq\C$. On the other
hand, if $\m\subset A$ is an ideal such that $A/\m\simeq\C$, then $\m$ is
a maximal proper ideal and $\m=\ker\zeta$ for some character $\zeta$.

A commutative algebra $A$ is called {\it local\/} if it has a unique ideal $\m$
such that $A/\m\simeq\C$. In other words, a commutative algebra $A$ is local
if it has a unique character. It is easy to see that any proper ideal of
the local algebra $A$ is contained in the ideal $\m$.

It is known that any finite-dimensional commutative algebra $A$ is isomorphic
to a direct sum of local algebras, and the local summands are in bijection
with characters of $A$.

Let $A$ be a commutative algebra. A bilinear form $(\,{,}\,):A\otimes A\to\C$
is called {\it invariant\/} if $(ab,c)=(a,bc)$ for all $a,b,c\in A$.

A finite-dimensional commutative algebra $A$ which admits an invariant
nondegenerate symmetric bilinear form ${(\,{,}\,):A\otimes A\to\C}$ is called
a {\it Frobenius algebra\/}. It is easy to see that distinct local summands of
a Frobenius algebra are orthogonal.

The following properties of Frobenius algebras will be useful.

\begin{lem} [\cite{MTV6}]
\label{direct}
A finite direct sum of Frobenius algebras is a Frobenius algebra.
\qed
\end{lem}

Let $A$ be a Frobenius algebra. Let $I\subset A$ be a subspace.
Denote by $I^\perp\subset A$ the orthogonal complement to $I$.
Then $\dim I+\dim I^\perp=\dim A$, and the subspace $I$ is an ideal
if and only if $I^\perp$ is an ideal.

Let $A_0$ be a local Frobenius algebra with maximal ideal $\m\subset A_0$.
Then $\m^\perp$ is a one-dimen\-sional ideal. Let $m^\perp\in\m^\perp$ be
an element such that $(1,m^\perp)=1$.

\begin{lem}[\cite{MTV6}]
\label{inverse}
Any nonzero ideal $I\subset A_0$ contains $\m^\perp$.
\qed
\end{lem}

For a subset $I\subset A$ define its annihilator as
$\Ann\,I\,=\,\{a\in A,\ |\ aI=0\}$. The annihilator $\Ann I$ is an ideal.

\begin{lem}[\cite{MTV6}]
\label{Ann=perp}
Let $A$ be a Frobenius algebra and $I\subset A$ an ideal.
Then $\Ann I=I^\perp$. In particular, $\dim I+\dim\>\Ann I=\dim A$.
\qed
\end{lem}

For any ideal $I\subset A$, the regular action of $A$ on itself induces
an action of $A/I$ on $\Ann I$.

\begin{lem}[\cite{MTV6}]
\label{coreg}
The $A/I$-module $\Ann I$ is isomorphic to the coregular representation of
$A/I$. In particular, the image of $A/I$ in $\End(\Ann I)$ is a maximal
commutative subalgebra.
\qed
\end{lem}

Let $P_1,\dots,P_m$ be polynomials in variables $x_1,\dots,x_m$.
Denote by $I$ the ideal in 
\linebreak
$\C[x_1,\dots,x_m]\kern-1em$ \kern1em
generated by $P_1,\dots,P_m$.

\begin{lem}[\cite{MTV6}]
\label{resform}
If the algebra $\C[x_1,\dots,x_m]/I$ is nonzero and finite-dimensional,
then it is a Frobenius algebra.
\qed
\end{lem}

The last lemma has the following generalization. Let $\C_T(x_1,\dots,x_m)$
be the algebra of rational functions in $x_1,\dots,x_m$, regular at points of a
subset $T\subset\C^m$. Denote by $I_T$ the ideal in $\C_T(x_1,\dots,x_m)$
generated by $P_1,\dots,P_m$.

\begin{lem}
\label{resformloc}
Assume that the solution set to the system of equations
$$
P_1(x_1,\dots,x_m) = \dots = P_m(x_1,\dots,x_m) = 0 
$$
is finite and lies in $T$. Then the algebra
$\C_T(x_1,\dots,x_m)/I_T$ is a Frobenius algebra.
\qed
\end{lem}

\subsection{Wronski map}
\label{wronski}

Let $X$ be a point of $\Om_{\bs\la}$. Define
\vvn.3>
\beq
\label{wronsk of X}
\Wr_X(u)\,=\,\Wr(f_1(u,X),\dots,f_N(u,X))\,,
\vv.3>
\eeq
where $f_1(u,X),\dots,f_N(u,X)$ are given by \Ref{Def of space}.
Define the {\it Wronski map\/} $\pi:\Om_{\bs\la}\to\C^n$ by
\,$X\mapsto\bs a=(a_1,\dots,a_n)$ \,if
\be
\Wr_X(u)\,=\,e^{\>\sum_{i=1}^N K_iu}\!\prod_{1\le i<j\le N}(K_j-K_i)
\ \Bigl(u^n+\sum_{s=1}^n (-1)^s\>a_s\,u^{n-s}\Bigr)\,.
\ee

\sskip
For $\bs a\in\C^n$, let ${I^\O_\lba}$ be the ideal
in $\O_{\bs\la}\>$ generated by the elements $\Sig_s-a_s$, $s=1,\dots, n$,
where $\Sig_1,\dots,\Sig_n$ are defined by \Ref{Wr coef}. The quotient algebra
\vvn.3>
\beq
\label{Olaa}
\O_\lba\>=\,\O_{\bs\la}/I^\O_\lba
\vv.3>
\eeq
is the scheme-theoretic fiber of the Wronski map. We call it
the {\it algebra of functions on the preimage\/} $\pi^{-1}(\bs a)$.

\goodbreak
\begin{lem}\label{local Wr}
\strut
\begin{enumerate}
\item [(i)] The algebra $\O_\lba$ is a finite-dimensional
commutative associative unital algebra and $\dim_\C\,\O_\lba$
does not depend on $\bs a$.
\item [(ii)]
The algebra $\O_\lba$ is a Frobenius algebra.
\end{enumerate}
\end{lem}
\begin{proof}
The Wronski map is a polynomial map of finite degree, see
Propositions 4.2 and 3.1 in \cite{MTV5}.
This implies part (i) of the lemma and the fact that $\O_\lba$
is a direct sum of local algebras.
The dimension of $\O_\lba$ is the degree of the Wronski map and
the local summands correspond to the points of the set $\pi^{-1}(\bs a)$.
The algebra $\O_\lba$ is Frobenius by Lemma~\ref{resform}.
\end{proof}

\section{Intersection $\OmLb$ and algebra $\OLb$}
\label{schubert}
\subsection{Intersection $\OmLb$}
For $b\in \C$ and a partition $\bs\mu$ of $n$ with at most $N$ parts,
denote by $\Om_{\bs\mu}(b)$ the variety of all spaces of quasi-exponentials
$X\in\Om_{\bs\la}$ such that for every $i=1,\dots,N$ there exists a function
$g(u)\in X$ with zero of order $\mu_i+N-i$ at $b$.

Let $\bs\La=(\bs\la^{(1)},\dots,\bs\la^{(k)})$ be a sequence of partitions
with at most $N$ parts such that $\sum_{s=1}^k|\bs\la^{(s)}|=n$.
Denote $n_s=|\bs\la^{(s)}|$.
Let $\bs b=(b_1,\dots,b_k)$ be a sequence of distinct complex numbers.

Consider the intersection
\vvn-.8>
\beq
\label{Omega}
\OmLb\,=\,\bigcap_{s=1}^k\,\Om_{\bs\la^{(s)}}(b_s)\,.
\eeq

Given a space of quasi-exponentials $X\subset\Om_{\bs\la}$, denote by
$\D_X$ the monic scalar differential operator of order $N$ with kernel $X$.
The operator $\D_X$ equals the operator $\D^\O_{\bs\la}$, see~\Ref{DOla},
computed at $X$.

\begin{lem}
\label{lem on intersection}
A space of quasi-exponentials $X\subset\Oml$ is a point of\/ $\OmLb$ if and
only if the singular points of the operator\/ $\D_X$ are at $b_1,\dots,b_k$
and $\infty$ only, the singular points at $b_1,\dots,b_k$ are regular,
and the exponents at\/ $b_s$, $s=1,\dots,k$, are equal to
$\la_N^{(s)},\,\la_{N-1}^{(s)}+1,\,\dots\,,\alb\la_1^{(s)}+N-1$.
\qed
\end{lem}

\begin{lem}
\label{lem inclus}
Let $\bs b=(b_1,\dots,b_k)$, and\/ $\bs n=(n_1,\dots,n_k)$ be as at the beginning of this
section. Let the numbers 
$\bs a=(a_1,\dots,a_n)$ be related to $\bs b$ and $\bs n$
as in \Ref{ab}.
Then\/ $\OmLb\subset\pi^{-1}(\bs a)$.
In particular, the set $\,\OmLb$ is finite.
\qed
\end{lem}

Let $\Q_{\bs\la}$ be the field of fractions of $\O_{\bs\la}\>$,
and $\>\Q_\Llb\subset\Q_{\bs\la}\,$ the subring of elements
regular at all points of $\OmLb$.

Consider the $N\,{\times}\,N$ matrices $M_1,\dots,M_k$ with entries in
$\O_{\bs\la}$\>,
\be
(M_s)_{ij}\>=\,\frac1{(\la^{(s)}_j+N-j)!}\,\biggl(\<\Bigl(
\frac d{du}\Bigr)^{\la^{(s)}_j+N-j}f_i(u)\biggr)\bigg|_{u=b_s}\,.
\ee
The values of $M_1,\dots,M_k$ at any point of $\Om_\Lbl$
are matrices invertible over $\C$. Therefore, the inverse matrices
$M_1^{-1},\dots,M_k^{-1}$ exist as matrices with entries in
$\Q_\Llb$.

\sskip
Introduce the elements $g_{ijs}\in\Q_\Llb$\>, \,$i=1,\dots,N$,
$j=0,\dots,d_1$, $s=1,\dots,k$, by the rule
\beq
\label{gijs}
\sum_{j=0}^{d_1}\,g_{ijs}\,(u-b_s)^j\,=\,
\sum_{m=1}^N\,(M_s^{-1})_{im}\,f_m(u)\,.
\eeq
Clearly, $g_{i,\la^{(s)}_j+N-j,s}=\dl_{ij}$ for all $i,j=1,\dots,N$, and
$s=1,\dots,k$.

For each $s=1,\dots,k$, let $J^{\Q,s}_\Llb$ be
the ideal in $\Q_\Llb\>$ generated by the elements $g_{ijs}$,
$\,i=1,\dots,N$, $\,j=0,\dots,\la^{(s)}_i+N-i-1$, and
$J^\Q_\Llb=\sum_{s=1}^k J^{\Q,s}_\Llb$.
Note that the number of generators of the ideal
$J^\Q_\Llb$ equals $n$.

\sskip
The quotient algebra
\vvn-.4>
\beq
\label{OX}
\O_\Llb\,=\,\Q_\Llb/J^\Q_\Llb
\eeq
is the scheme-theoretic intersection of varieties $\Om_{\bs\la^{(s)}},\,
s=1,\dots,k$. We call it the
{\it algebra of functions on\/} $\OmLb\,$.

\begin{lem}
\label{frobenius}
The algebra $\OmLb$ is a Frobenius algebra.
\end{lem}
\begin{proof}
The claim follows from Lemma~\ref{resformloc}.
\end{proof}

It is known from Schubert calculus that
\beq
\label{dimO}
\dim\O_\Llb\,=\,\dim\>
(\otimes_{s=1}^kL_{\bs\la^{(s)}})_{\bs\la}\,,
\eeq
see \cite[Lemma 3.6 and Proposition 3.7]{MTV8}.

\subsection{Algebra $\OLb$ as a quotient of $\O_{\bs\la}$}
Consider the differential operator
\beq
\label{DOlat}
\Dt^\O_{\bs\la}=\,\rdet
\left(\begin{matrix} f_1(u) & f_1'(u) &\dots & f_1^{(N)}(u) \\
f_2(u) & f_2'(u) &\dots & f_2^{(N)}(u) \\ \dots & \dots &\dots & \dots \\
1 & \der &\dots & \der^N
\end{matrix}\right).
\eeq
It is a differential operator in the variable $u$ whose coefficients are
polynomials in $u$ with coefficients in $\O_{\bs\la}$,
\vvn-.6>
\beq
\label{DOt}
\Dt^\O_{\bs\la}=\,\sum_{i=0}^N\,G_i(u)\,\der^{N-i}\>.
\eeq
Clearly, $G_i(u)=0$ \>for \>$i>n$, and \,$\deg G_i \leq n$, otherwise.
We also have
\vvn.3>
\begin{align*}
G_0(u)\,&{}=\,\Wr(f_1(u),\dots,f_N(u))\,,
\\[4pt]
G_i(u)\,&{}=\,\Wr(f_1(u),\dots,f_N(u))\>F_i(u)\,,\qquad i=1,\dots,N\,,\kern-3em
\\[-12pt]
\end{align*}

Introduce the elements $G_{ijs}\in\O_{\bs\la}$\>, \,$i=0,\dots,N$,
$j=0,\dots,n-i$, $s=1,\dots,k$, by the rule
\beq
\label{Gijs}
G_i(u)\,=\,\sum_{j=0}^{n}\,G_{ijs}\,(u-b_s)^j\,.
\eeq
Define the {\it indicial polynomial\/ $\chi_s^\O(\al)$ at\/ $b_s$}
by the formula
\be
\chi_s^\O(\al)\,=\,
\sum_{i=0}^{N}\,G_{i,n_s-i,s}\prod_{j=0}^{N-i-1} (\al-j)\,.
\ee
It is a polynomial of degree $N$ in the variable $\al$ with coefficients in
$\O_{\bs\la}$.

\begin{lem}
\label{chis}
For a complex number $r$, the element\/ $\chi_s^\O(r)$ is invertible
in\/ $\Q_\Llb$ provided\/ $r\ne\la^{(s)}_j+N-j$ \,for all\/ $j=1,\dots,N$.
\end{lem}
\begin{proof}
An element of $\Q_\Llb$ is invertible if and only if its value
at any point of $\Om_\Lbl$ is nonzero. Now the claim follows from
Lemmas~\ref{lem on intersection} and~\ref{lem inclus}.
\end{proof}

For each $s=1,\dots,k$, let $I^{\Q,s}_\Llb$ be
the ideal in $\Q_\Llb\>$ generated by the elements $G_{ijs}$,
$\,i=0,\dots,N$, $\,0\le j<n_s-i$, and the coefficients of the polynomials
\beq
\label{chisO}
\chi_s^\O(\al)\,-\,\prod_{\fratop{r=1}{r\ne s}}^k\,(b_s-b_r)^{n_r}\,
\prod_{l=1}^N\,(\al-\la_l^{(s)}-N+l)\,,\qquad s=1,\dots,k\,.\kern-3em
\eeq
Denote \,$I^\Q_\Llb=\sum_{s=1}^k I^{\Q,s}_\Llb\,$.

\begin{lem}[Lemma 5.5 in \cite{MTV6}]
\label{ideals}
For any $s=1,\dots,k$, the ideals\/ $I^{\Q,s}_\Llb$ and\/ $J^{\Q,s}_\Llb$
coincide.
\qed
\end{lem}

Let $I^\O_\Llb$ be the ideal in $\O_{\bs\la}\>$ generated by the elements
$G_{ijs}$, $\,i=0,\dots,N$, \>$s=1,\dots,k$, $\,0\le j<n_s-i$, and
the coefficients of polynomials~\Ref{chisO}.

\begin{prop}
\label{OIX}
The algebra \>$\O_\Llb$ is isomorphic to the quotient algebra\/
\>$\O_{\bs\la}/I^\O_\Llb\>$.
\end{prop}
\begin{proof}
By Lemma~\ref{ideals}, the ideals\/ $I^\Q_\Llb$ and\/ $J^\Q_\Llb$ coincide,
so the algebra \>$\O_\Llb$ is isomorphic to the quotient algebra
\>$\Q_\Llb/I^\Q_\Llb\>$. By Lemma~\ref{lem on intersection}, the algebraic
set defined by the ideal $I^\O_\Llb$ equals $\Om_\Lbl$\,. The set $\Om_\Lbl$
is finite by Lemma~\ref{lem inclus}. Therefore, the quotient algebras
$\Q_\Llb/I^\Q_\Llb\>$ and $\O_{\bs\la}/I^\O_\Llb\>$ are isomorphic.
\end{proof}

\subsection{Algebra $\O_\Llb$ as a quotient of $\O_\lba$}
\label{barI}
Recall that $\O_\lba=\O_{\bs\la}/{I^\O_\lba}$ is the algebra of functions on
$\pi^{-1}(\bs a)$, see~\Ref{Olaa}. For an element $F\in\O_{\bs\la}$,
we denote by $\bat F$ the projection of $F$ to the quotient algebra $\O_\lba$.

Define the {\it indicial polynomial\/ $\bat\chi_s^\O(\al)$ at\/ $b_s$}
by the formula
\be
\bat\chi_s^\O(\al)\,=\,
\sum_{i=0}^{N}\,\bat G_{i,n_s-i,s}\prod_{j=0}^{N-i-1}(\al-j)\,.
\ee
Let $\bat I^\O_\Llb$ be the ideal in $\O_\lba\>$ generated by the elements
$\bat G_{ijs}$, $\,i=1,\dots,N$, $\,s=1,\dots,k$, $\,0\le j<n_s-i$,
and the coefficients of the polynomials
\be
\bat\chi_s^\O(\al)\,-\,\prod_{\fratop{r=1}{r\ne s}}^k\,(b_s-b_r)^{n_r}\,
\prod_{l=1}^N\,(\al-\la_l^{(s)}-N+l)\,,\qquad s=1,\dots,k\,.\kern-3em
\ee

\begin{prop}
\label{scheme}
The algebra\/ $\O_\Llb$ is isomorphic to the quotient
algebra\/ $\O_\lba/\bat I^\O_\Llb$.
\end{prop}
\begin{proof}
It is easy to see that the elements $G_{0js}$, $j=0,\dots,n_s-1$,
$s=1,\dots,k$, generate the ideal $I^\O_\lba$ in $\O_{\bs\la}$\>. Moreover,
the projection of the ideal $I^\O_\Llb\subset\O_{\bs\la}$ \>to $\O_\lba$ equals
$\bat I^\O_\Llb$\,. Hence, the claim follows from Proposition~\ref{OIX}.
\end{proof}

Recall that the ideal $\Ann(\bat I^\O_\Llb)\subset\O_\lba$ is naturally
an $\O_\Llb$-module.

\begin{cor}
\label{AnnI}
The\/ $\O_\Llb$-module\/ $\Ann(\bat I^\O_\Llb)$ is isomorphic to the coregular
representation of $\O_\Llb$ on the dual space $(\O_\Llb)^*$.
\end{cor}
\begin{proof}
The statement follows from Lemmas~\ref{frobenius} and~\ref{coreg}.
\end{proof}

\section{Three isomorphisms}
\label{iso sec}
\subsection{Auxiliary lemma}
Let $\bs\la$ be a partition of $n$ with at most $N$ parts.
Recall that given a space of quasi-exponentials
$X\in \Oml$,
we denote by $\D_X$ the monic scalar differential operator of order $N$ with
kernel $X$.

Let $M$ be a $\glnt$-module $M$ and $v$ an eigenvector of the Bethe algebra
$\B\subset\Uglnt$ acting on $M$. Then for any coefficient $B_i(u)$ of
the universal differential operator $\D^\B$ we have $B_i(u)v=h_i(u)v$,
where $h_i(u)$ is a scalar series. We call the scalar differential operator
\beq
\label{DBv}
\D^\B_v\>=\,\der^N+\>\sum_{i=1}^N\>h_i(u)\,\der^{N-i}
\eeq
the differential {\it  operator associated with\/} $v$.

We consider $\C^n$ with the symmetric group $S_n$
action defined by permutation of coordinates.

\begin{lem}
\label{generic}
There exist a Zariski open $S_n$-invariant subset $\Tht$ of $\C^n$ and
a Zariski open subset $\Xi$ of $\Oml$ with the following
properties.
\begin{enumerate}
\item[(i)]
For any $(b_1,\dots,b_n)\in\Tht$, there exists a basis of
$\bigl(\otimes_{s=1}^nV(b_s)\bigr)_{\bs\la}$ such that every basis
vector $v$ is an eigenvector of the Bethe algebra and $\D^\B_v=\D_X$ for some
$X\in\Xi$. Moreover, different basis vectors correspond to different points
of\/ $\Xi$.
\item[(ii)]
For any $X\in\Xi$\,, if\/ $b_1,\dots,b_n$ are all roots of the Wronskian\/
$\Wr_X$, then $(b_1,\dots,b_n)\in\Tht$, and there exists a unique up to
proportionality vector $v\in\bigl(\otimes_{s=1}^nV(b_s)\bigr)_{\bs\la}$
such that $v$ is an eigenvector of the Bethe algebra with $\D^\B_v=\D_X$.
\end{enumerate}
\end{lem}

\begin{proof}
The basis in part (i) is constructed by the Bethe ansatz method as
in Section 10
of \cite{MTV4}. The equality $\D^\B_v=\D_X$ is proved in~\cite{MTV1}.
The existence of an eigenvector $v$ in part (ii) for generic
$X\subset\Om_\lab$, is proved as in Section 10 of \cite{MTV4}.
\end{proof}

\begin{cor}\label{degree}
The degree of the Wronski map equals\/
\,$\dim\>(V^{\otimes n})_{\bs\la}$.
\qed
\end{cor}

\subsection{Isomorphism of algebras $\O_{\bs\la}$ and $\B_{\bs\la}$}
Consider the $\B$-module $(\V^S)_{\bs\la}$. Denote $(\V^S)_{\bs\la}$ by
$\Ml$ and  the Bethe
algebra associated with $\VSl$ by  $\B_{\bs\la}$.

Consider the map
\vvn-.1>
\be
\tau_{\bs\la}:\O_{\bs\la}\to\B_{\bs\la}\,,\qquad
F_{ij}\mapsto\hat B_{ij}\,,
\vv.2>
\ee
where the elements $F_{ij}\in\O_{\bs\la}$ are defined by~\Ref{Fi} and
\>$\hat B_{ij}\in\B_{\bs\la}$ are the images of the elements $B_{ij}\in\B$,
defined by~\Ref{Bi}.

\begin{thm}
\label{first}
The map $\tau_{\bs\la}$ is a well-defined isomorphism of algebras.
\end{thm}
\begin{proof}
Let a polynomial $R(F_{ij})$ in generators $F_{ij}$ be equal to zero
in $\O_{\bs\la}$. Let us prove that the corresponding polynomial
$R(\hat B_{ij})$ is equal to zero in the $\B_{\bs\la}$. Indeed,
$R(\hat B_{ij})$ is a polynomial in $z_1,\dots,z_n$ with values in
$\End\bigl((V^{\otimes n})_{\bs\la}\bigr)$. Let $\Tht$ be the set,
introduced in Lemma~\ref{generic}, and $(b_1,\dots,b_n)\in\Tht$.
Then by part (i) of Lemma~\ref{generic}, the value of the polynomial
$R(\hat B_{ij})$ at $z_1=b_1,\dots,z_n=b_n$ equals zero. Hence, the polynomial
$R(\hat B_{ij})$ equals zero identically and the map $\tau_{\bs\la}$ is
well-defined.

Let a polynomial $R(F_{ij})$ in generators $F_{ij}$ be a nonzero element of
$\O_{\bs\la}$. Then the value of $R(F_{ij})$ at a generic point
$X\in\Om_\lab(\infty)$ is not equal to zero. Then by part (ii) of
Lemma~\ref{generic}, the polynomial $R(\hat B_{ij})$ is not identically equal
to zero. Therefore, the map $\tau_{\bs\la}$ is injective.

Since the elements $\hat B_{ij}$ generate the algebra $\B_{\bs\la}\,$,
the map $\tau_{\bs\la}$ is surjective.
\end{proof}

The algebra $\C[z_1,\dots,z_n]^S$ is embedded into the algebra $\B_{\bs\la}$
as the subalgebra of operators of multiplication by symmetric polynomials,
see Lemmas~\ref{Uz} and formula~\Ref{Bii}. The algebra $\C[z_1,\dots,z_n]^S$
is embedded into the algebra $\O_{\bs\la}$\>, the elementary symmetric
polynomials $\si_1(\bs z),\dots,\si_n(\bs z)$ being mapped to the elements
$\Sig_1,\dots,\Sig_n$, defined by \Ref{Wr coef}. These embeddings give
the algebras $\B_{\bs\la}$ and $\O_{\bs\la}$ the structure of
$\C[z_1,\dots,z_n]^S\<$-modules.

\begin{lem}\label{symm OK}
The map $\,\tau_{\bs\la}:\O_{\bs\la}\to\B_{\bs\la}$ is
an isomorphism of $\,\C[z_1,\dots,z_n]^S\<$-modules, that is,
$\tau_{\bs\la}(\Sig_i)=\si_i(\bs z)$ \>for all\/ \>$i=1,\dots,n$.
\end{lem}
\begin{proof}
The claim follows from the fact that
\beq
\label{F1}
F_1(u)\,=\,-\,\frac{\Wr'(f_1(u),\dots,f_N(u))}{\Wr(f_1(u),\dots,f_N(u))}\ ,
\eeq
where $\,'$ denotes the derivative with respect to $u$, and from
formula~\Ref{B1}.
\end{proof}

\begin{lem}
\label{taudeg}
For any homogeneous element\/ \>$F\in\O_{\bs\la}$\>, the degrees of homogeneous
components of\/ \,$\tau_{\bs\la}(F)\in\B_{\bs\la}$ do not exceed\/ \,$\deg F$.
\end{lem}
\begin{proof}
It suffices to prove the claim for the generators $f_{ij}\in \Ol$. In that case,
the statement follows from formula~\Ref{fij} and Lemma~\ref{B deg} by induction
with respect to $j$, starting from $j=1$.
\end{proof}

Given a vector $v\in\M_{\bs\la}$, consider a linear map
\be
\mu_v:\O_{\bs\la}\to\M_{\bs\la}\,,\qquad F\mapsto\tau_{\bs\la}(F)\,v\,.
\ee

\begin{lem}
\label{muinject}
If\/ $v\in\M_{\bs\la}$ is nonzero, then the map \>$\mu_v$ is injective.
\end{lem}
\begin{proof}
The algebra $\O_{\bs\la}$ is a free polynomial algebra containing
the subalgebra $\C[z_1,\dots,z_n]^S$. By part (i) of Lemma~\ref{local Wr},
the quotient algebra $\O_{\bs\la}/\C[z_1,\dots,z_n]^S$ is finite-dimensional.
The kernel of $\mu_v$ is an ideal in $\B_{\bs\la}$ which has zero
intersection with $\C[z_1,\dots,z_n]^S$ and, therefore, is the zero ideal.
\end{proof}

The graded character of $\V^S_{\bs\la}$ is given by formual
\Ref{forMula}.  Fix a nonzero vector $v\in\V^S_{\bs\la}$ of degree
0. Such a vector is unique up to multiplication by a nonzero number.
Then the map $\mu_v$ will be denoted by $\mu_{\bs\la}$.

\begin{thm}
\label{first1}
The map\/ \>$\mu_{\bs\la}:\O_{\bs\la}\to
\Ml$ is a vector isomorphism.
This isomorphism preserves the degree of elements.
The maps\/
$\tau_{\bs\la}$ and\/ $\mu_{\bs\la}$ intertwine the action of multiplication
operators on $\O_{\bs\la}$ and the action of the Bethe algebra $\B_{\bs\la}$
on $\Ml$, that is, for any $F,G\in\O_{\bs\la}$, we have
\beq
\label{mutau}
\mu_{\bs\la}(FG)\,=\,\tau_{\bs\la}(F)\,\mu_{\bs\la}(G)\,.
\eeq
In other words, the maps\/ $\tau_{\bs\la}$ and\/ $\mu_{\bs\la}$ give
an isomorphism of the regular representation of\/ $\O_{\bs\la}$ and
the\/ $\B_{\bs\la}$-module $\Ml$.
\end{thm}
\begin{proof}
The map $\mu_{\bs\la}$ is injective by Lemma~\ref{muinject}.
The map $\mu_{\bs\la}$ does not increase the degree by Lemma \ref{taudeg}.
The graded  characters of $\Ol$ and $\Ml$ are the same by Lemmas
\ref{char O} and \ref{lem on char of VSl}. 
Hence, the map $\mu_{\bs\la}$ is surjective.
Formula~\Ref{mutau} follows from Theorem~\ref{first}.
\end{proof}

\subsection{Isomorphism of algebras $\O_\lba$ and $\B_\lba$}
Let $\bs a=(a_1,\dots,a_N)$ be a sequence of complex numbers.
Let distinct complex numbers $b_1,\dots,b_k$ and integers $n_1,\dots,n_k$
be given by~\Ref{ab}.

Let $I^\B_\lba\subset\B_{\bs\la}$ be the ideal generated by
the elements $\si_i(\bs z)-a_i$, $i=1,\dots,n$. Consider the subspace
$\,I^\M_\lba=I^\B_\lba\>\M_{\bs\la}$,
where $I^\V_{\bs a}$ is given by~\Ref{IVa}. Recall that the ideal
$I^\O_\lba$ is defined in Section~\ref{wronski}.

\begin{lem}
\label{identify}
We have
\vvn.2>
\be
\tau_{\bs\la}({I^\O_\lba})={I^\B_\lba}\,,\qquad
\mu_{\bs\la}({I^\O_\lba})={I^\M_\lba}\,,\qquad
\B_\lba=\B_{\bs\la}/{I^\B_\lba}\,,\qquad
\M_\lba=\M_{\bs\la}/{I^\M_\lba}\,.
\vv.4>
\ee
\end{lem}
\begin{proof}
The lemma follows from Theorems~\ref{first}, \ref{first1} and
Lemmas~\ref{symm OK}, \ref{factor=weyl}.
\end{proof}

By Lemma~\ref{identify},
the maps $\tau_{\bs\la}$ and $\mu_{\bs\la}$ induce the maps
\beq
\label{taumu}
\tau_\lba:\O_\lba\to\B_\lba\,,\qquad
\mu_\lba:\O_\lba\to\M_\lba\,.
\eeq

\begin{thm}
\label{second}
The map $\tau_\lba$ is an isomorphism of algebras. The map $\mu_\lba$ is
an isomorphism of vector spaces. The maps\/ $\tau_\lba$ and\/ $\mu_\lba$
intertwine the action of multiplication operators on $\O_\lba$ and the action
of the Bethe algebra $\B_\lba$ on $\M_\lba$, that is, for any $F,G\in\O_\lba$,
we have
\vvn-.2>
\be
\mu_\lba(FG)\,=\,\tau_\lba(F)\,\mu_\lba(G)\,.
\vv.3>
\ee
In other words, the maps\/ $\tau_\lba$ and\/ $\mu_\lba$ give an isomorphism of
the regular representation of\/ $\O_\lba$ and the\/ $\B_\lba$-module $\M_\lba$.
\end{thm}
\begin{proof}
The theorem follows from Theorems~\ref{first}, \ref{first1}
and Lemma~\ref{identify}.
\end{proof}

\begin{rem}
By Lemma~\ref{local Wr}, the algebra $\O_\lba$ is Frobenius.
Therefore, its regular and coregular representations are isomorphic.
\end{rem}

\subsection{Isomorphism of algebras $\O_\Llb$ and $\B_\Llb$}
Let $\bs\La=(\bs\la^{(1)},\dots,\bs\la^{(k)})$ be a sequence of partitions
with at most $N$ parts such that $|\bs\la^{(s)}|=n_s$ for all $s=1,\dots,k$.

Consider the $\B$-module $(\otimes_{s=1}^kL_{\bs\la^{(s)}}(b_s))_{\bs\la}$. 
Denote $(\otimes_{s=1}^kL_{\bs\la^{(s)}}(b_s))_{\bs\la}$  by
$\Mlb$ and  the Bethe
algebra associated with $(\otimes_{s=1}^kL_{\bs\la^{(s)}}(b_s))_{\bs\la}$  
 by  $\Blb$.

\medskip
We begin with an observation. Let $A$ be an associative unital algebra,
and let $L,M$ be $A$-modules such that $L$ is isomorphic to a subquotient of
$M$. Denote by $A_L$ and $A_M$ the images of $A$ in $\End(L)$ and $\End(M)$,
respectively, and by $\pi_L:A\to A_L$, $\pi_M:A\to A_M$ the corresponding
epimorphisms. Then, there exists a unique epimorphism $\pi_{ML}:A_M\to A_L$
such that $\pi_L=\pi_{ML}\circ\pi_M$.

Applying this observation to the Bethe algebra $\B$ and $\B$-modules
$\M_{\bs\la}\,$, $\M_\lba\,$, $\M_\Llb\,$, we get a chain of epimorphisms
$\B\to\B_{\bs\la}\to\B_\lba\to\B_\Llb\,$. In particular, each module over
a smaller Bethe algebra is naturally a module over a bigger Bethe algebra.

\medskip

For any element $F\in\B_{\bs\la}$, we denote by $\bat F$ the projection of
$F$ to the algebra $\B_\lba$.

Let $C_1(u),\dots,C_N(u)$ be the polynomials with coefficients in
$\B_{\bs\la}\,$, defined in Lemma~\ref{capelli}. Introduce the elements
$C_{ijs}\in\B_{\bs\la}$ for $i=1,\dots,N$, $j=0,\dots,n$, $s=1,\dots,k$,
by the rule
\be
\sum_{j=0}^{n}\,C_{ijs}\,(u-b_s)^j\,=\,C_i(u)\,.
\ee
In addition, let $\>\bat C_{0js}$\>, \,$j=0,\dots,n$, $s=1,\dots,k$,
be the numbers such that
\be
\sum_{j=0}^n\,\bat C_{0js}\,(u-b_s)^j\,=\,\prod_{r=1}^k\,(u-b_r)^{n_r}\,.
\ee
Define the {\it indicial polynomial\/ $\bat\chi_s^\B(\al)$ at\/ $b_s$}
by the formula
\be
\bat\chi_s^\B(\al)\,=\,
\sum_{i=0}^{N}\,\bat C_{i,n_s-i,s}\prod_{j=0}^{N-i-1} (\al-j)\,.
\ee
It is a polynomial of degree $N$ in the variable $\al$ with coefficients in
$\B_\lba$.

Let $I^\B_\Llb$ be the ideal in $\B_\lba\>$ generated by the elements
$\>\bat C_{ijs}$, $\,i=1,\dots,N$, $\,s=1,\dots,k$, $\,0\le j<n_s-i$,
and the coefficients of the polynomials
\beq
\label{chiB}
\bat\chi_s^\B(\al)\,-\,\prod_{\fratop{r=1}{r\ne s}}^k\,(b_s-b_r)^{n_r}\,
\prod_{l=1}^N\,(\al-\la_l^{(s)}-N+l)\,,\qquad s=1,\dots,k\,.\kern-3em
\eeq

\begin{lem}
\label{zero}
The ideal\/ $I^\B_\Llb$ belongs to the kernel of the projection\/
$\>\B_\lba\to\B_\Llb\,$.
\end{lem}
\begin{proof}
The statement follows from Lemma~\ref{Apol} and Corollary~\ref{Apol2}.
\end{proof}

Hence, the projection $\B_\lba\to\B_\Llb$ descends to an epimorphism
\beq
\label{pi}
\pi_\Llb:\B_\lba/I^\B_\Llb\,\to\,\B_\Llb\,,
\eeq
which makes $\M_\Llb$ into a $\>\B_\lba/I^\B_\Llb$-module.

\sskip
Denote $\,\ker\>(I^\B_\Llb)=\{\>v\in\M_\lba\ |\ I^\B_\Llb\>v=0\>\}\,$.
Clearly, $\,\ker\>(I^\B_\Llb)$ is a $\>\B_\lba$-submodule of $\M_\lba$.

\begin{prop}
\label{kerI}
The $\>\B_\lba/I^\B_\Llb$-modules $\,\ker\>(I^\B_\Llb)$
and\/ $\M_\Llb$ are isomorphic.
\end{prop}
\noindent
The proposition is proved in Section~\ref{main proof}.

\medskip
Let $\bat I^\O_\Llb\subset\O_\lba$ be the ideal defined in Section~\ref{barI}.
\vvn.1>
Clearly, the map $\tau_\lba:\O_\lba\to\B_\lba$ sends $\bat I^\O_\Llb$ to
$I^\B_\Llb$. By Lemma~\ref{scheme}, the maps $\tau_\lba$ and $\pi_\Llb$
induce the homomorphism
\vvn.1>
\be
\tau_\Llb:\O_\Llb\,\to\,\B_\Llb\,.
\vv.1>
\ee

By Theorem~\ref{second}, the map $\mu_\lba:\O_\lba\to\M_\lba$ sends
$\Ann(\bat I^\O_\Llb)\subset\O_\lba$ to $\ker\>(I^\B_\Llb)$\>.
The vector spaces $\Ann(\bat I^\O_\Llb)$ and $(\O_\Llb)^*$ are isomorphic
by Corollary~\ref{AnnI}. Hence, Proposition~\ref{kerI} yields that the map
$\mu_\lba$ induces a bijective linear map
\be
\mu_\Llb:(\O_\Llb)^*\to\,\M_\Llb\,.
\vv.1>
\ee

\sskip
For any $F\in\O_\Llb$, denote by $F^*\in\End\bigl((\O_\Llb)^*\bigr)$
the operator, dual to the operator of multiplication by $F$ on $\O_\Llb$.

\begin{thm}
\label{third}
The map $\tau_\Llb$ is an isomorphism of algebras. The maps\/ $\tau_\Llb$ and\/
$\mu_\Llb$ intertwine the action of the operators on $(\O_\Llb)^*$, dual to
the multiplication operators on $\O_\Llb$\>, and the action of the Bethe
algebra $\B_\Llb$ on $\M_\Llb$, that is, for any $F\in\O_\lba$ and
$G\in(\O_\Llb)$\>, we have
\vvn-.2>
\be
\mu_\Llb(F^*G)\,=\,\tau_\Llb(F)\,\mu_\Llb(G)\,.
\vv.3>
\ee
In other words, the maps\/ $\tau_\Llb$ and\/ $\mu_\Llb$ give an isomorphism of
the coregular representation of\/ $\O_\Llb$ on the dual space $(\O_\Llb)^*$
and the $\B_\Llb$-module $\M_\Llb$.
\end{thm}
\begin{proof}
By Lemma~\ref{scheme}, the isomorphism $\tau_\lba:\O_\lba\to\B_\lba$ induces
the isomorphism
\be
\tau_\Llb:\O_\Llb\,\to\,\B_\lba/I^\B_\Llb\,.
\ee
so the maps $\tau_\Llb$ and $\mu_\lba$ give an isomorphism of
the $\O_\Llb$-module $\Ann(\bat I^\O_\Llb)$ and the $\B_\lba/I^\B_\Llb$-module
$\ker\>(I^\B_\Llb)$, see Theorem~\ref{second}.

By Lemma~\ref{coreg}, the $\O_\Llb$-module $\Ann(\bat I^\O_\Llb)$ is
isomorphic to the coregular representation of $\O_\Llb$ on the dual space
$(\O_\Llb)^*$. In particular, it is faithful.
Therefore, the $\B_\lba/I^\B_\Llb$-module $\ker\>(I^\B_\Llb)$ is faithful.
By Proposition~\ref{kerI}, the $\>\B_\lba/I^\B_\Llb$-module
$M_{\bs\la,\bs\la,\bs b}\,$, isomorphic to $\ker\>(I^\B_\Llb)$, is faithful
too, which implies that the map $\pi_\Llb:\B_\lba/I^\B_\Llb\to\B_\Llb$
is an isomorphism of algebras. The theorem follows.
\end{proof}

\begin{rem}
By Lemma~\ref{frobenius}, the algebra $\O_\Llb$ is Frobenius.
Therefore, its coregular and regular representations are isomorphic.
\end{rem}

\subsection{Proof of Proposition \ref{kerI}}
\label{main proof}
We begin the proof with an elementary auxiliary lemma. Let $M$ be
a finite-dimensional vector space, $U\subset M$ a subspace, and $E\in\End(M)$.
\begin{lem}
\label{EU}
Let\/ $EM\subset U$, and the restriction of\/ $E$ to\/ $U$ is invertible in\/
$\End(U)$. Then\/ $EU=U$ and $\,M=U\oplus\ker E$.
\qed
\end{lem}

Let $W_m$ be a Weyl module, see Section~\ref{secweyl}, and $\bs\mu$ a partition
with at most $N$ parts such that $|\bs\mu|=m$. Recall that $W_m$ is a graded
vector space, the grading of $W_m$ is defined in Lemma~\ref{weyl}.

Given a homogeneous vector $w\in(W_m)_{\bs\mu}^{sing}$, let $\L_w(b)$ be the
$\glnt$-submodule of $W_m(b)$ generated by the vector $v$. The space $\L_w(b)$
is graded. Denote by $\L_w^=(b)$ and $\L_w^>(b)$ the subspaces of $\L_w(b)$
spanned by homogeneous vectors of degree $\>\deg w$ and of degree strictly
greater than $\>\deg w$, respectively. The subspace $\L_w^=(b)$ is
a $\gln$-submodule of $\L_w(b)$ isomorphic to the irreducible $\gln$-module
$L_{\bs\mu}$. The subspace $\L_w^>(b)$ is a $\glnt$-submodule of $\L_w(b)$,
and the $\glnt$-module $\L_w(b)/\L_w^>(b)$ is isomorphic to the evaluation
module $L_{\bs\mu}(b)$. If $v$ has the largest degree possible for vectors
in $(W_m)_{\bs\mu}^{sing}$, then $\L_w^>(b)$, considered as a $\gln$-module,
does not contain $L_{\bs\mu}$.

For any $s=1,\dots,k$, pick up a homogeneous vector
$w_s\in(W_{n_s})_{\bs\la^{(s)}}^{sing}$ of the largest possible degree.
Let $\L_{\bs w}(\bs b)$ be the $\glnt$-submodule of
$\otimes_{s=1}^k W_{n_s}(b_s)$ generated by the vector $\otimes_{s=1}^k w_s$.

Denote by $\L_{\bs w}^=(\bs b)$ and $\L_{\bs w}^>(\bs b)$
the following subspaces of $\L_{\bs w}(\bs b)$:
\begin{gather*}
\L_{\bs w}^=(\bs b)\,=\,\otimes_{s=1}^k\L_{w_s}^=(b_s)\,,
\\
\L_{\bs w}^>(\bs b)\,=\,\sum_{s=1}^k\,\L_{w_1}(b_1)\otimes\dots\otimes
\L_{w_s}^>(b_s)\otimes\dots\otimes\L_{w_k}(b_k)\,.
\end{gather*}
The subspace $\L_{\bs w}^>(\bs b)$ is a $\glnt$-submodule
of $\L_{\bs w}(\bs b)$, and the $\glnt$-module
$\L_{\bs w}(\bs b)/\L_{\bs w}^>(\bs b)$ is isomorphic to the tensor
product of evaluation modules $\otimes_{s=1}^kL_{\bs\la^{(s)}}(b_s)$.

The space $\otimes_{s=1}^k W_{n_s}$ has the second $\glnt$-module structure,
denoted $\gr\bigl(\otimes_{s=1}^k W_{n_s}(b_s)\bigr)$, which was introduced
at the end of Section~\ref{secweyl}. The subspace $\L_{\bs w}(\bs b)$ is
a $\glnt$-submodule of $\gr\bigl(\otimes_{s=1}^k W_{n_s}(b_s)\bigr)$,
isomorphic to a direct sum of irreducible $\glnt$-modules of the form
$\otimes_{s=1}^k L_{\bs\mu^{(s)}}(b_s)$, where $|\bs\mu^{(s)}|=n_s$,
$\,s=1,\dots,k$, see Lemmas~\ref{grWmb} and~\ref{grtensor}, and
$(\bs\mu^{(1)},\dots,\bs\mu^{(k)})\ne(\bs\la^{(1)},\dots,\bs\la^{(k)})$
for any term of the sum.

The subspace $\M_\lba=(\otimes_{s=1}^k W_{n_s}(b_s))_{\bs\la}$
is invariant under the action of the Bethe algebra ${\B\subset\Uglnt}$.
This makes it  a $\B$-module, which we call the standard $\B$-module
structure on $\M_\lba$. The $\B$-module $\M_\lba$ contains
the submodules
$\M_\Llb^{\bs w}=(\L_{\bs w}(\bs b))_{\bs\la}$
and $\M_\Llb^{\bs w,>}=(\L_{\bs w}^>(\bs b))_{\bs\la}$,
and the subspace
$\M_\Llb^{\bs w,=}=(\L_{\bs w}^=(\bs b))_{\bs\la}$.
As vector spaces,
$\M_\Llb^{\bs w}=\M_\Llb^{\bs w,=}
\oplus\M_\Llb^{\bs w,>}\,$. The $\B$-modules
$\M_\Llb^{\bs w}/\M_\Llb^{\bs w,>}$
and $\M_\Llb$ are isomorphic.

The space $\M_\lba$ has another $\B$-module structure,
inherited from the $\glnt$-module structure of
$\gr\bigl(\otimes_{s=1}^k W_{n_s}(b_s)\bigr)$.
We denote the new structure $\gr\M_\lba\,$.
The subspaces $\M_\Llb^{\bs w}\,$,
$\M_\Llb^{\bs w,=}\,$, $\M_\Llb^{\bs w,>}$
are $\B$-submodules of the $\B$-module $\gr\M_\lba\,$. The submodule
$\M_\Llb^{\bs w,=}\subset\gr\M_\lba$ is isomorphic
to the $\B$-module $\M_\Llb\,$, and the submodule
$\M_\Llb^{\bs w,>}\subset\gr\M_\lba$ is isomorphic
to a direct sum of $\B$-modules of the form $\M_{\bs\Mu,\bs\la,\bs b}\,$, where
$\bs\Mu=(\bs\mu^{(1)},\dots,\bs\mu^{(k)})$, \,$|\bs\mu^{(s)}|=n_s$,
$\,s=1,\dots,k$, and $\bs\Mu\ne\bs\La$ for any term of the sum.

\sskip
In the picture described above, we can regard all $\B$-modules involved
as $\B_\lba$-modules.

\sskip
For any $F\in\B_\lba$, we denote by $\gr F\in\End(\M_\lba)$ the linear operator
corresponding to the action of $F$ on $\gr\M_\lba$\>. The map $F\mapsto\gr F$
is an algebra homomorphism.

Let complex numbers $c_1,\dots,c_k$, $\al_1,\dots,\al_k$ be such that
\be
\sum_{s=1}^k\,c_s\>\Bigl(\,\prod_{i=1}^N\,(\al_s-\mu_i^{(s)}-N+i)-
\prod_{i=1}^N\,(\al_s-\la_i^{(s)}-N+i)\Bigr)\,\ne\,0\,,
\ee
for any sequence of partitions
$(\bs\mu^{(1)},\dots,\bs\mu^{(k)})\ne(\bs\la^{(1)},\dots,\bs\la^{(k)})\,$. Introduce
\vvn.1>
\be
E\,=\,\sum_{s=1}^k c_s\>\Bigl(\bat\chi_s^\B(\al_s)-
\prod_{\fratop{r=1}{r\ne s}}^k\,(b_s-b_r)^{n_r}\,
\prod_{i=1}^N\,(\al_s-\la_i^{(s)}-N+i)\Bigr)\,\in I^\B_\Llb\,,
\vv-.2>
\ee
where $\bat\chi_s^\B$ is the indicial polynomial~\Ref{chiB}.
\vvn.1>
With respect to the standard $\B$-module structure on $\M_\lba$,
we have $E\>\M_\Llb^{\bs w}\subset\M_\Llb^{\bs w,>}$.

\begin{lem}
The restriction of\/ $E$ to $\M_\Llb^{\bs w,>}$ is invertible in
$\End(\M_\Llb^{\bs w,>})$.
\end{lem}
\begin{proof}
Lemma~\ref{zero} implies that the projection of $E$ to $\B_\Llb$ equals zero,
and the projection of $E$ to $\B_{\bs\Mu,\bs\la,\bs b}$ with $\bs\Mu\ne\bs\La$
is invertible. This means that the restriction of the operator $\gr E$
to $\M_\Llb^{\bs w,>}$ is invertible in $\End(\M_\Llb^{\bs w,>})$.
Therefore, the restriction of\/ $E$ to $\M_\Llb^{\bs w,>}$ is invertible
in $\End(\M_\Llb^{\bs w,>})$.
\end{proof}

Denote $\ker_\Llb^{\bs w}E\>=\>\ker E\,\cap\,\M_\Llb^{\bs w}\,$.
By Lemma~\ref{EU}, the canonical projection
\be
\M_\Llb^{\bs w}\to
\M_\Llb^{\bs w}/\M_\Llb^{\bs w,>}\simeq
\M_\Llb
\ee
induces an isomorphism $\ker_\Llb^{\bs w}E\to\M_\Llb$ of vector spaces.
\vvn.1>
Since the algebra $\B_\lba$ is commutative, the subspace $\ker_\Llb^{\bs w}E$
\vvn.1>
is a $\B$-submodule, and the map $\ker_\Llb^{\bs w}E\to\M_\Llb$ is
an isomorphism of $\B_\lba$-modules.

Lemma~\ref{zero} implies that elements of the ideal $I^\B_\Llb$ act on
$\M_\Llb$ by zero. Hence, they act by zero on $\ker_\Llb^{\bs w}E$, that is,
$\ker_\Llb^{\bs w}E\subset \ker\>(I^\B_\Llb)\,$. On the other hand, we have
\be
\dim\>\ker\>(I^\B_\Llb)\,=\,\dim\>\Ann(\bat I^\O_\Llb)\,=\,
\dim\>\O_\Llb\,=\,\dim\>\M_\Llb\,=\,\dim\>\ker_\Llb^{\bs w}E\,,
\ee
see Theorem~\ref{second}, Corollary~\ref{AnnI} and formula~\Ref{dimO},
which yields \,$\ker_\Llb^{\bs w}E=\ker\>(I^\B_\Llb)\,$.
Proposition~\ref{kerI} is proved.
\qed

\begin{rem}
Note that formula~\Ref{dimO} is the key ingredient of the proof.
\end{rem}

\section{Applications}
\label{app sec}
\subsection{Action of the Bethe algebra in a tensor product of evaluation
modules}
In this section we summarize obtained results in a way
independent from  the main part of the paper.
For convenience, we recall some definitions.

Let $\bs K=(K_1,\dots,K_N)$ be a sequence of distinct complex numbers.
The Bethe algebra $\B$ is a commutative subalgebra of $\Uglnt$,
defined in Section~\ref{secbethe} with the help of this sequence.
 It is generated by the elements
$B_{ij}$, $i=1,\dots,N$, $j\in\Z_{\ge i}$, given by
formula~\Ref{Bi}. The Bethe algebra depends on the choice of $\bs K$.
In the remainder of the paper we will denote this algebra by $\B_{\bs K}$.

\sskip
If $M$ is a $\B_{\bs K}$-module and $\xi:\B_{\bs K}\to\C$ a homomorphism, then
the eigenspace of the
$\B_{\bs K}$-action on $M$ corresponding to $\,\xi\,$ is defined
as $\,\bigcap_{F\in\B_{\bs K}}
\ker(F|_M-\xi(F))$ and the generalized eigenspace of
the $\B_{\bs K}$-action on $M$ corresponding to $\,\xi\,$ is defined as
$\,\bigcap_{F\in\B_{\bs K}}\bigl(\,\bigcup_{m=1}^\infty\ker(F|_M-\xi(F))^m\bigr)$.

\sskip
For a partition $\bs\la$ with at most $N$ parts, let $L_{\bs\la}$ be
the irreducible finite-dimensional $\gln$-module of highest weight $\bs\la$.

Let $\bs\la^{(1)},\dots,\bs\la^{(k)}$ be partitions with at most $N$ parts,
$b_1,\dots,b_k$ distinct complex numbers. We are interested in the action
of the Bethe algebra $\B_{\bs K}$ on the tensor product
$\otimes_{s=1}^kL_{\bs\la^{(s)}}(b_s)$ of evaluation $\glnt$-modules.

Since $\B_{\bs K}$ commutes with the subalgebra $U(\h)\subset\Uglnt$,
the action of $\B_{\bs K}$ preserves 
the weight subspaces of $\otimes_{s=1}^kL_{\bs\la^{(s)}}(b_s)$.

\sskip
Denote $\bs\La=(\bs\la^{(1)},\dots,\bs\la^{(k)})$. Given a partition $\bs\la$
with at most $N$ parts such that $|\bs\la|=\sum_{s=1}^k|\bs\la^{(s)}|\,$,
let $\Dlb$ be the set of all monic  differential operators of
order $N$,
\beq
\label{mcD}
\D\,=\,\der^N+\>\sum_{i=1}^N\,h_i^\D(u)\,\der^{N-i}\,,
\eeq
where $\der=d/du$, with the following properties:
\begin{enumerate}
\flati
\item[\=a]
The singular points of $\D$ are at\/ $b_1,\dots,b_k$ and $\infty$ only.
\item[\=b]
The exponents of $\D$ at\/ \>$b_s$\>, \,$s=1,\dots,k$, are equal to
$\,\la_N^{(s)},\,\la_{N-1}^{(s)}+1,\alb\,\dots\,,\la_1^{(s)}+N-1\,$.
\item[\=c]
The kernel of $\D$ is generated by quasi-exponentials of the form
\bea
g_i(u)\,=\,e^{K_iu}\,(u^{\la_i}+g_{i1}u^{\la_i-1}+\dots+g_{i\la_i})\,,
\qquad i=1,\dots,N\,,
\eea
where $g_{ij}$ are suitable complex numbers.
\end{enumerate}

A differential operator $\D$ belongs to the set $\Dlb$ if and only if
the kernel of $\D$ is a point of the intersection 
$\,\OmLb$\>, see Lemma~\ref{lem on intersection}.

Denote $n_s=|\bs\la^{(s)}|\,$, $\,s=1,\dots,k$, \,and $\,n=\sum_{s=1}^kn_s\,$.

\begin{thm}
\label{BL}
The action of the Bethe algebra $\,\B_{\bs K}$ on
$\otimes_{s=1}^kL_{\bs\la^{(s)}}(b_s)$ has the following properties.
\begin{enumerate}
\item[(i)]
For every $i=1,\dots,N$, the action of the series\/ $B_i(u)$ is given by
\vvn.1>
the power series expansion in\/ $u^{-1}\!$ of a rational function of the form
\vvn.2>
$\,A_i(u)\prod_{s=1}^k(u-b_s)^{-n_s}$,
where $A_i(u)$ is a polynomial of degree $\,n$ with coefficients in
$\,\End(\otimes_{s=1}^kL_{\bs\la^{(s)}})$.

\sskip
\item[(ii)]
The image of $\,\B_{\bs K}$ in
$\,\End (\otimes_{s=1}^kL_{\bs\la^{(s)}})$
\vvn.1>
is a maximal commutative subalgebra of dimension
$\,\dim\>\otimes_{s=1}^kL_{\bs\la^{(s)}}$.

\sskip
\item[(iii)]
Each eigenspace of the action of $\,\B_{\bs K}$ is one-dimensional.

\sskip
\item[(iv)] Each generalized eigenspace of the action of $\,\B_{\bs K}$ is
generated over $\,\B_{\bs K}$ by one vector.

\sskip
\item[(v)] The eigenspaces of the action of $\,\B_{\bs K}$ on
$(\otimes_{s=1}^kL_{\bs\la^{(s)}}(b_s))_{\bs\la}$ are in a one-to-one
correspondence with differential operators from $\>\Dlb\,$. Moreover,
if\/ $\D$ is the differential operator, corresponding to an eigenspace, then
the coefficients of the series\/ $h_i^\D(u)$ are the eigenvalues of the action
of the respective coefficients of the series\/ $\,B_i(u)$.

\sskip
\item[(vi)] The eigenspaces of the action of $\,\B_{\bs K}$ on
$(\otimes_{s=1}^kL_{\bs\la^{(s)}}(b_s))_{\bs\la}$ are in a one-to-one
correspondence with points of the intersection $\,\OmLb$\>, given by~\Ref{Omega}.
\end{enumerate}
\end{thm}
\begin{proof}
The first property follows from Corollary~\ref{Apol2}. The other properties
follow from Theorem~\ref{third}, Lemma~\ref{lem on intersection}, and standard
facts about the coregular representations of Frobenius algebras given in
Section~\ref{comalg}.
\end{proof}

\begin{cor}
\label{card}
The following three statements are equivalent.
\begin{enumerate}
\item[(i)]
The action of the Bethe algebra $\>\B_{\bs K}$ on
$(\otimes_{s=1}^kL_{\bs\la^{(s)}}(b_s))_{\bs\la}$ is diagonalizable.

\sskip
\item[(ii)]
The set $\>\Dlb\,$ consists of\/
$\,\dim\>(\otimes_{s=1}^kL_{\bs\la^{(s)}})_{\bs\la}$ distinct points.

\sskip
\item[(iii)]
The set $\,\OmLb$ consists of\/
$\,\dim\>(\otimes_{s=1}^kL_{\bs\la^{(s)}})_{\bs\la}$ distinct points.
\qed
\end{enumerate}
\end{cor}

The intersection  $\OmLb$ is
{\it transversal\/} if the scheme-theoretic intersection
$\O_\Llb$ is a direct sum of one-dimensional algebras.

\begin{cor}
\label{cor 2}
The action of the Bethe algebra $\>\B_{\bs K}$ on
\vvn.1>
$(\otimes_{s=1}^kL_{\bs\la^{(s)}}(b_s))_{\bs\la}$ is diagonalizable,
if and only the  $\,\OmLb$
is transversal.
\end{cor}
\begin{proof}
The algebra $\O_\Llb$ is a direct sums of local algebras, each local summand
corresponding to a point of the set $\OmLb\,$. Therefore, the intersection
$\,\OmLb$ is transversal if and only if the dimension of
$\O_\Llb$ equals the cardinality of $\OmLb\,$. Corollary~\ref{card}
completes the proof.
\end{proof}

\begin{cor}
\label{real}
Let $K_1,\dots,K_N$ be distinct real numbers.
Let $b_1,\dots,b_k$ be distinct real numbers. Then
\begin{enumerate}
\item[(i)] The set $\>\Dlb\,$ consists of\/
$\,\dim\>(\otimes_{s=1}^kL_{\bs\la^{(s)}})_{\bs\la}$ distinct points;
\item[(ii)]
The intersection  $\,\OmLb$ consists of\/
$\dim\>(\otimes_{s=1}^kL_{\bs\la^{(s)}})_{\bs\la}$ distinct points
and is transversal.
\end{enumerate}
\end{cor}
\begin{proof}
If $K_1,\dots,K_N$ are distinct real numbers and
$b_1,\dots,b_k$ are distinct real numbers, then the action of the Bethe
\vvn.1>
algebra $\>\B_{\bs K}$ on $(\otimes_{s=1}^kL_{\bs\la^{(s)}}(b_s))_{\bs\la}$
is diagonalizable, see~\cite{MTV1}, cf. \cite{MTV2}.
\end{proof}

Results similar to Theorem~\ref{BL} and Corollary~\ref{card}
hold for the action of the Bethe algebra $\B_{\bs K}$ on
the $\glnt$-module $\otimes_{s=1}^k W_{n_s}(b_s)$, the Weyl module associated
with $\bs n=(n_1,\dots,n_k)$ and $\bs b=(b_1,\dots,b_k)$, defined in
Section~\ref{secweyl}. The action of $\B_{\bs K}$ preserves 
 the weight subspaces
of $\otimes_{s=1}^k W_{n_s}(b_s)$. 

\sskip
Recall that $V$ denotes the irreducible $\gln$-module of highest weight
$(1,0,\dots,0)$, which is the vector representation of $\gln$.

\sskip
Denote by $\Dlnb$ the set of all monic differential operators $\D$ of order $N$
with the following properties.

\begin{enumerate}
\flati
\item[\=a]
The kernel of $\D$ is generated by quasi-exponentials of the form
\bea
g_i(u)\,=\,e^{K_iu}\,(u^{\la_i}+g_{i1}u^{\la_i-1}+\dots+g_{i\la_i})\,,
\qquad i=1,\dots,N\,,
\eea
where $\la_1+\dots+\la_N=n$ and $g_{ij}$ are suitable complex numbers.

\item[\=b]
The first coefficient $h_1^\D(u)$ of $\D$, see~\Ref{mcD},
equals $\>\sum_{i=1}^NK_i\,+\,\sum_{s=1}^k n_s\>(b_s-u)^{-1}$.
\end{enumerate}

\sskip\noindent
If $\D\in\Dlnb$, then $\D$ is a differential operator with singular
points at\/ $b_1,\dots,b_k$ and $\infty$ only.

\medskip

Denote by $\>\Omn$ the set of all $N$-dimensional spaces of
quasi-exponentials with a basis of the form
\bea
g_i(u)\,=\,e^{K_iu}\,(u^{\la_i}+g_{i1}u^{\la_i-1}+\dots+g_{i\la_i})\,,
\qquad i=1,\dots,N\,,
\eea
and such that
\bea
\Wr(g_1(u),\dots,g_N(u))\ =\ 
e^{\>\sum_{i=1}^N K_iu}\!\prod_{1\le i<j\le N}(K_j-K_i)
\prod_{s=1}^k(u-b_s)^{n_s}\ .
\eea
A differential operator $\D$ belongs to the set $\Dlnb$ if and only if
the kernel of $\D$ belongs to the set $\>\Omn$.

\begin{thm}
\label{BW}
The action of the Bethe algebra $\,\B_{\bs K}$ on
$\otimes_{s=1}^k W_{n_s}(b_s)$ has the following properties.
\begin{enumerate}
\item[(i)]
For every $i=1,\dots,N$, the action of the series\/ $B_i(u)$ is given by
\vvn.1>
the power series expansion in\/ $u^{-1}\!$ of a rational function of the form
\vvn.2>
$\,A_i(u)\prod_{s=1}^k(u-b_s)^{-n_s}$,
where $A_i(u)$ is a polynomial of degree $\,n$ with coefficients in
$\,\End\bigl(\otimes_{s=1}^k W_{n_s}\bigr)$.

\sskip
\item[(ii)]
The image of $\,\B_{\bs K}$ in $\,\End\bigl(\otimes_{s=1}^k W_{n_s}\bigr)$
\vvn.1>
is a maximal commutative subalgebra of dimension
$\,\dim\>V^{\otimes n}$.

\sskip
\item[(iii)]
Each eigenspace of the action of $\,\B_{\bs K}$ is one-dimensional.

\sskip
\item[(iv)] Each generalized eigenspace of the action of $\,\B_{\bs K}$ is
generated over $\,\B_{\bs K}$ by one vector.

\sskip
\item[(v)] The eigenspaces of the action of $\,\B_{\bs K}$ on
$\otimes_{s=1}^k W_{n_s}(b_s)$ are in a one-to-one correspondence
with differential operators from $\>\Dlnb\,$. Moreover,
if\/ $\D$ is the differential operator, corresponding to an eigenspace, then
the coefficients of the series\/ $h_i^\D(u)$ are the eigenvalues of the action
of the respective coefficients of the series\/ $\,B_i(u)$.

\sskip
\item[(vi)] The eigenspaces of the action of $\,\B_{\bs K}$ on
\vvn.1>
$\otimes_{s=1}^k W_{n_s}(b_s)$ are in a one-to-one
correspondence with spaces of polynomials from\/ $\Wrnbi$.
\end{enumerate}
\end{thm}
\begin{proof}
The first property follows from Lemmas~\ref{factor=weyl} and~\ref{capelli}.
The other properties follow from Theorem~\ref{second}, formulae~\Ref{DOla}
and~\Ref{F1}, and standard facts about the regular representations of Frobenius
algebras given in Section~\ref{comalg}.
\end{proof}

\begin{cor}
\label{cardW}
The following three statements are equivalent.
\begin{enumerate}
\item[(i)]
The action of the Bethe algebra $\>\B_{\bs K}$ on
$\otimes_{s=1}^k W_{n_s}(b_s)$ is diagonalizable.

\sskip
\item[(ii)]
The set $\>\Dlnb\,$ consists of\/
$\,\dim\>V^{\otimes n}$ distinct points.

\sskip
\item[(iii)]
The set $\,\Omn$ consists of\/
$\,\dim\>V^{\otimes n}$ distinct points.
\qed
\end{enumerate}
\end{cor}

\section{Completeness of Bethe ansatz}
\label{Completness of Bethe}

\subsection{Generic points of $\Omlb$}

  Let
$\Omlb$ be the affine $n+N$-dimensional space with coordinates
$g_{ij},\,i=1,\dots,N,\,j=1,\dots,\la_i$, and $k_1,\dots,k_N$.
We identify points \>$Y\in\Omlb$ with $N$-dimensional complex
vector spaces generated by quasi-exponentials
\beq
\label{def of space}
g_i(u,Y)\,=\,e^{k_i(Y)_iu}\,
(u^{\la_i}+g_{i1}(Y)u^{\la_i-1}+\dots+g_{i\la_i}(Y))\,,
\qquad i=1,\dots,N\,.
\eeq

Let \>$Y\in\Omlb$ be a point with
distinct cordinates $k_1(Y),\dots,k_N(Y)$.
Denote by $\B_Y \subset \Uglnt$ the Bethe algebra constructed
in Section \ref{secbethe}  with the help of the sequence 
${\bs K} = (K_1,\dots,K_N)$ where $K_i=k_i(Y)$ for all $i$.

\medskip

For  \>$Y\in\Omlb$, 
introduce the polynomials
$\{y_0(u)\>,\,y_{1}(u)\>,\,\dots\,,\,y_{N-1}(u)\}$, \,by the formula
\vvn.3>
\be
y_a(u)\, e^{\sum_{i=a+1}^N k_i(Y)u}\!\!
\prod_{a<i<j\leq N}(k_i(Y)-
k_j(Y))\,
\,=\,\Wr( g_{a+1}(u,Y),\dots,g_{N}(u,Y))\,,
\ee
for $a = 0,\dots,N$.
Set
\vvn-.6>
\beq
\label{sequence l}
l_a\,=\sum_{b=a+1}^N \la_b\,,\qquad a=0,\dots,N\,.
\eeq
Clearly, $l_0=|\bs\la|$ \,and \,$l_N=0$.

\sskip
For each $a=0,\dots,N-1$, the polynomial $y_a(u)$ is a monic polynomial of
degree $l_a$.
Denote $t_1^{(a)},\dots,t_{l_a}^{(a)}$ the roots of the polynomial $y_a(u)$,
and put
\beq
\label{t z}
\bs t_{Y}\,=\,\,(t_1^{(0)},\dots,t_{l_0}^{(0)},\,\dots\,,
t_1^{(N-1)},\dots,t_{l_{N-1}}^{(N-1)})\,.
\eeq
We say that $\bs t_{Y}$ are the {\it root coordinates} of $Y$.

We say that $Y\in\Omlb$ is {\it generic} if all roots of the polynomials
$y_0(u)\>,\,y_1(u)\>,\,\,\dots\,,\,y_{N-1}(u)$ are simple and for each
$a=1,\dots,N-1$, the polynomials $y_{a-1}(u)$ and $y_a(u)$ do not have common
roots.

\medskip

If $Y$ is generic, then the root coordinates $\bs t_Y$ satisfy the Bethe ansatz
equations \cite{MV1}, cf. \cite{MTV4}:
\bea
\label{BAE Gaudin}
\sum_{j'=1}^{l_{a-1}}\frac 1{t^{(a)}_j - t^{(a-1)}_{j'}}\;-\,
\sum_{\satop{j'=1}{j'\neq j}}^{l_a}\frac 2{t^{(a)}_j - t^{(a)}_{j'}}\;+\,
\sum_{j'=1}^{l_{a+1}}\frac 1{t^{(a)}_j - t^{(a+1)}_{j'}}\;=\,K_{a+1}-K_a\,.
\vv-.2>
\eea
Here the equations are labeled by $a=1,\dots,N-1$, \,$j=1,\dots,l_a$.

\sskip
Conversely, if \,$\bs t\,=\,(t_1^{(0)},\dots,t_{l_0}^{(0)},\,\dots\,,
t_1^{(N-1)},\dots,t_{l_{N-1}}^{(N-1)})$ satisfy the Bethe ansatz equations,
then there exists a unique $Y\in \Omlb$ such that $Y$ is generic and $\bs t$ are
its root coordinates. This $Y$ is determined by the following
construction, see~\cite{MV1},  cf. \cite{MTV4}. Set
\be
\chi^a(u, \bs t)\,=\,K_a\,+\,\sum_{j=1}^{l_{a-1}}\,\frac1{u-t^{(a-1)}_j}\;-\,
\sum_{i=1}^{l_a}\,\frac 1 {u- t^{(a)}_j}\;,\qquad a=1,\dots,N\,.
\ee
Then the monic differential operator $\D_Y$
with kernel $Y$ is given by the 
formula:
\be
\D_Y\,=\,\bigl(\der -\chi^1(u,\bs t)\bigr)\,\dots\,
\bigl(\der-\chi^N(u,\bs t)\bigr)\,.
\ee
Clearly, the operator $\D_Y$ determines $Y$

\begin{lem}
\label{lem on generic pts}
Generic points form a Zariski open subset of\/ $\Omlb$.
\end{lem}

The lemma follows from Theorem 10.5.1 in \cite{MTV4}.

\subsection{Universal weight function}
Let $\bs\la$ be a partition with at most $N$ parts. Let $l_0,\dots,l_N$
be the numbers defined in \Ref{sequence l}. Denote $n=l_0$\>,
\,$l=l_1+\dots+l_{N-1}$ \>and \,$\bs l=(l_1,\dots,l_{N-1})$.

\sskip
Consider the weight subspace $(V^{\otimes n})_{\bs\la}$ of the $n$-th tensor
power of the vector representation of $\glN$ and the space $\C^{l+n}$ with
coordinates \,$\bs t\,=\,(t_1^{(0)},\dots,t_{l_0}^{(0)},\,\dots\,,
t_1^{(N-1)},\dots,t_{l_{N-1}}^{(N-1)})$.

In this section we remind the construction of a rational map
$\omega:\C^{l+n}\to (V^{\otimes n})_{\bs\la}$,
called the {\it universal weight function}, see~\cite{SV}.

A basis of $V^{\otimes n}$ is formed by the vectors
\vvn.1>
\be
e_J\>v\,=\,e_{j_1,1}\>v_+\otimes \dots\otimes e_{j_n,1}\>v_+\,,
\vv.2>
\ee
where $J=(j_1,\dots,j_n)$ and $1\leq j_s\leq N$ for $s=1,\dots,N$. A basis
of $(V^{\otimes n})_{\bs\la}$ is formed by the vectors $e_J\>v$ such that
$\#\{s\ |\ j_s>i\}\,=\,l_i$ for every $i=1,\dots,N-1$.
Such a $J$ will be called $\bs l$-admissible.

The universal weight function has the form
\vvn.3>
\be
\omega(\bs t)\,=\,\sum_J\,\omega_J(\bs t)\,e_Jv
\ee
where the sum is over the set of all $\bs l$-admissible $J$,
and the function $\omega_J(\bs t)$ is defined below.

\sskip
For an admissible $J$, \>define \,$S(J)=\{s\ |\ j_s>1 \}$\>,
and for \>$i=1,\ldots,N-1$, define
\vvn.3>
\be
S_i(J)\,=\,\{\,s\ |\ 1\le s\le n\,,\ \ 1\le i<j_s\,\}\,.
\vv-.2>
\ee
Then $|\>S_i(J)\>|\,=\,l_i$.

\sskip
Let $B(J)$ be the set of sequences \,$\bs\beta=(\beta_1,\dots,\beta_{N-1})$
of bijections \,$\beta_i:S_i(J)\to\{1,\dots,l_i\}$, \>$i=1,\dots,N-1$.
Then $|\>B(J)\>|\>=\prod_{a=1}^{N-1}l_a!$~.

\sskip
For $s\in S(J)$ and $\bs\beta\in B(J)$, introduce the rational function
\be
\omega_{s,\bs\beta}(\bs t)\,=\,\frac1{t^{(1)}_{\beta_1(s)}-t^{(0)}_s}\;
\prod_{i=2}^{j_1-1}\frac1{t^{(i)}_{\beta_i(s)}-t^{(i-1)}_{\beta_{i-1}(s)}}\
\ee
and define
\be
\omega_J(\bs t)\,=\,
\sum_{\bs\beta\in B(J)}\,\prod_{s\in S(J)}\,\omega_{s,\bs\beta} \ .
\ee

\begin{example}
Let $n=2$ \,and \,$\bs l=(1,1,0,\dots,0)$. Then
\vvn.5>
\be
\omega(\bs t)\,=\,\frac 1{(t_1^{(2)}-t_1^{(1)})\>
(t_1^{(1)}-t_1^{(0)})}\ e_{3,1}v_+\otimes v_+\,+\;
\frac 1{(t_1^{(2)}-t_1^{(1)})\>(t_1^{(1)}-t_2^{(0)})}\ v_+\otimes e_{3,1}v_+\;.
\vv.5>
\ee
\end{example}

\begin{thm}
\label{thm X to Vn}
Let $Y\in \Omlb$ be a generic point with root coordinates $\bs t_Y$.
Consider the value $\omega(\bs t_Y)$ of the universal weight function
$\omega:\C^{l+n}\to(V^{\otimes n})_{\bs\la}$ at\/ $\bs t_Y$. Consider
$V^{\otimes n}$ as the $\glNt$-module $\otimes_{s=1}^nV(t_s^{(0)})$\>.
Consider the Bethe algebra $\B_Y\subset \Uglnt$.
Then
the vector $\omega(\bs t_Y)$ is an eigenvector of the Bethe algebra
$\B_Y$, acting on $\otimes_{s=1}^nV(t_s^{(0)})$. Moreover,
$\D^{\mc B_Y}_{\omega(\bs t_Y)}=\D_Y$, where $\D^{\mc B_Y}_{\omega(\bs t_Y)}$
and $\D_Y$ are the differential operators associated with the eigenvector
$\omega(\bs t_Y)$ and the point $Y\in\Omlb$, respectively.

\end{thm}

The theorem is proved in \cite{MTV1}.

\subsection{Epimorphism $F_{\bs\la}$}
\label{Construction of B hom}
Let $\bs\la^{(1)},\dots,\bs\la^{(k)},\>\bs\la$ \,be partitions with at most $N$
parts such that $|\bs\la|=\sum_{s=1}^k|\bs\la^{(s)}|$, and $b_1,\dots,b_k$
distinct complex numbers. Denote \,$n=|\bs\la|$ \,and \,$n_s=\>|\bs\la^{(s)}|$,
\,$s=1,\dots,k$.

For $s=1,\dots,k$, \,let \,$F_s:V^{\otimes n_s}\to L_{\bs\la^{(s)}}$
be an epimorphism of $\gln$-modules. Then
\beq
\label{FF}
F_1\otimes\dots\otimes F_k:\otimes_{s=1}^kV(b_s)^{\otimes n_s}\,\to\,
\otimes_{s=1}^k L_{\bs\la^{(s)}}(b_s)
\eeq
is an epimorphism of $\glNt$-module, which induces an epimorphism
of $\B_Y$-modules
\bea
\label{B homo}
F : (\otimes_{s=1}^k V(b_s)^{\otimes n_s})_{\bs\la}\,\to\,
(\otimes_{s=1}^k L_{\bs\la^{(s)}}(b_s))_{\bs\la}\,,
\eea
for any $Y$ with distinct coordinates $k_1(Y),\dots,k_N(Y)$.

\subsection{Construction of an eigenvector from a differential operator}
\label{Main result}
Let $\D^0$ be an element of $\Dlb$. Let $Y^0$ be the kernel of $\D^0$.
Then $Y^0$ is a point of the cell $\Omlb$ and $K_i=k_i(Y^0)$
for all $i$. In particular, we have $\B_{\bs K}=\B_{Y^0}$.

\medskip
 Choose a germ of an algebraic curve
$Y(\ep)$ in $\Omlb$ such that $Y(0)=Y^0$ and $Y(\ep)$ are generic points of
$\Omlb$ for all nonzero $\ep$. Let $\bs t(\ep)$ be the root coordinates of
$Y(\ep)$. The algebraic functions $t_1^{(0)}(\ep),\dots,t_n^{(0)}(\ep)$ are
determined up to permutation. Order them in such a way that the first $n_1$ of
them tend to $b_1$ as $\ep\to 0$, the next $n_2$ coordinates tend to $b_2$, and
so on until the last $n_k$ coordinates tend to $b_k$.

For every nonzero $\ep$, the vector $v(\ep)=\omega(\bs t(\ep))$ belongs to
$(V^{\otimes n})_{\bs\la}$. This vector is an eigenvector of the Bethe
algebra $\B_{Y(\ep)}$, acting on $(\otimes_{s=1}^nV(t_s^{(0)}(\ep)))_{\bs\la}$,
and we have $\D^{\mc B_{Y(\ep)}}_{v(\ep)}=\D_{Y(\ep)}$, 
see Theorem~\ref{thm X to Vn}.

The vector $v(\ep)$ algebraically depends on $\ep$.
Let $v(\ep)=v_0\>\ep^{a_0}+v_1\>\ep^{a_1}+\dots{}$ be its Puiseux expansion,
where $v_0$ is the leading nonzero coefficient.

\begin{thm}
\label{main thm}
For a generic choice of the maps $F_1,\dots,F_k$, the vector\/ $F(v_0)$ is
nonzero. Moreover, $F(v_0)$ is an eigenvector of the Bethe algebra\/ \>$\B_{\bs K}$,
acting on $(\otimes_{s=1}^kL_{\bs\la^{(s)}}(b_s))_{\bs\la}$, and\/
$\D^{\B_{\bs K}}_{F(v_0)}=\D^0$.
\end{thm}

\begin{proof}
For any generator $B_{ij}\in\B_{Y(\ep)}$, 
the action of $B_{ij}$ on the $\Uglnt$-module
$\otimes_{s=1}^nV(t^{(0)}_s(\ep))$ determines an element of $\End(V^{\otimes n})$,
algebraically depending on $\ep$. 
Since for every nonzero $\ep$,
the vector $v(\ep)$ is an eigenvector of $\B_{Y(\ep)}$, acting on
$(\otimes_{s=1}^nV(t_s^{(0)}(\ep)))_{\bs\la}$, and
$\D^{\mc B}_{v(\ep)}=\D_{Y(\ep)}$, the vector $v_0$ is an eigenvector of 
$\B_{Y(0)}=\B_{\bs K}$,
acting on $(\otimes_{s=1}^kV(b_s)^{\otimes n_s})_{\bs\la}$, and
$\D^{\mc B_{\bs K}}_{v_0}=\D^0$.

The $\glnt$-module $\otimes_{s=1}^kV(b_s)^{\otimes n_s}$ is a direct sum of
irreducible $\glnt$-modules of the form $\otimes_{s=1}^k L_{\bs\mu^{(s)}}(b_s)$,
where $|\bs\mu^{(s)}|=n_s$, \,$s=1,\dots,k$. Since $\D^0\in\Dlb$,
the vector $v_0$ belongs to the component of type
$\otimes_{s=1}^kL_{\bs\la^{(s)}}(b_s)$. Therefore, for generic choice
of the maps $F_1,\dots,F_k$, the vector\/ $F(v_0)$ is nonzero.

Since the map $F_1\otimes\dots\otimes F_k$, see~\Ref{FF},
is a homomorphism of $\glnt$-modules, the vector $F(v_0)$ is
an eigenvector of the Bethe algebra \>$\B_{\bs K}$, acting on
$(\otimes_{s=1}^kL_{\bs\la^{(s)}}(b_s))_{\bs\la}$, and
\>$\D^{\B_{\bs K}}_{F(v_0)}=\D^0$.
\end{proof}

Given $\D\in\Dlb$\>, denote by \,$w(\D)$ the vector
\>$F(v_0)\in(\otimes_{s=1}^kL_{\bs\la^{(s)}}(b_s))_{\bs\la}$ constructed
from $\D$ in Section~\ref{Main result}. The vector $w(\D)$ is defined up to
multiplication by a nonzero number. The assignment \>$\D\mapsto w(\D)$ gives
the correspondence, which is inverse to the correspondence \>$v\mapsto\D^\B_v$
\,in part (v) of Theorem~\ref{BL}.

\subsection{Completeness of Bethe ansatz for $\glN$ Gaudin model}
\label{Sec Bethe ansatz}

The construction of the vector
$w(\D)\in(\otimes_{s=1}^kL_{\bs\la^{(s)}}(b_s))_{\bs\la}$ from
a differential operator $\D\in\Dlb$ can be viewed as a (generalized)
Bethe ansatz construction for the $\glN$ Gaudin model, cf.~the Bethe ansatz
constructions in~\cite{Ba}, \cite{RV}, \cite{MV1}, \cite{MV2}.

\begin{thm}
\label{cor 1}
If\/ $b_1,\dots,b_k$ are distinct real numbers and $ K_1,\dots,K_N$ 
are distinct real numbers, then the collection of vectors
\vvn.4>
\be
\{\>w(\D)\in (\otimes_{s=1}^kL_{\bs\la^{(s)}}(b_s))_{\bs\la}
\ |\ \D\in\Dlb\>\}
\vv.4>
\ee
is an eigenbasis of the action of the Bethe algebra $\B_{\bs K}$.
\end{thm}

The theorem follows from Theorem \ref{BL} and Corollaries \ref{card} 
and 
\ref{real} .

\end{document}